\documentclass{article}
\usepackage{amssymb}
\usepackage{amsmath,amsthm}
\usepackage[latin1]{inputenc}
\usepackage{cite}
\usepackage[dvips]{graphicx}
\usepackage{hyperref}
\usepackage{color}
\usepackage{setspace}
\usepackage{enumerate}
\usepackage{tkz-graph}
\usepackage{tikz}

\hypersetup{colorlinks=true}

\hypersetup{colorlinks=true, linkcolor=blue, citecolor=blue,urlcolor=blue}

 \newtheorem{remark}{Remark}

 \newtheorem{lemma}[remark]{Lemma}

 \newtheorem{theorem}[remark]{Theorem}
 \newtheorem{proposition}[remark]{Proposition}
 \newtheorem{corollary}[remark]{Corollary}
 
  \newtheorem{claim}[remark]{Claim}

\newcommand{\Sd}{\operatorname{Sd}}

\newcommand{\Perm}{\operatorname{Perm}}

\title{The Simultaneous Metric Dimension of Families Composed by Lexicographic Product Graphs}

\author{Yunior Ram\'{\i}rez-Cruz, Alejandro Estrada-Moreno and Juan A.
Rodr\'{\i}guez-Vel\'{a}zquez\\
{\small Departament d'Enginyeria Inform\`atica i Matem\`atiques,}\\
{\small Universitat Rovira i Virgili,}  {\small Av. Pa\"{\i}sos
Catalans 26, 43007 Tarragona, Spain.} \\{\small
yunior.ramirez\@@urv.cat, alejandro.estrada\@@urv.cat, juanalberto.rodriguez\@@urv.cat}}
\begin{document}
\maketitle

\begin{abstract}
Let ${\cal G}$ be a graph family defined on a common (labeled) vertex set $V$. 
A set   $S\subseteq V$ is said to be a simultaneous metric generator for ${\cal G}$ if for every $G\in {\cal G}$ and every pair of different vertices $u,v\in V$ there exists $s\in S$ such that $d_{G}(s,u)\ne d_{G}(s,v)$, where $d_{G}$ denotes the  geodesic distance.
A \emph{simultaneous adjacency generator} for ${\cal G}$ is a simultaneous metric generator under the metric $d_{G,2}(x,y)=\min\{d_{G}(x,y),2\}$.
A minimum cardinality simultaneous metric (adjacency) generator for ${\cal G}$ is a simultaneous metric (adjacency) basis, and its cardinality the simultaneous metric (adjacency) dimension of ${\cal G}$. Based on the simultaneous adjacency dimension, we study the simultaneous metric dimension of families composed by lexicographic product graphs.
\end{abstract}

\section{Introduction}
\label{sectIntro}

A generator of a metric space $(X,d)$ is a set $S\subset X$ of points in the space  with the property that every point of $X$  is uniquely determined by the distances from the elements of $S$. Given a simple and connected graph $G=(V,E)$, we consider the function $d_G:V\times V\rightarrow \mathbb{N}\cup \{0\}$, where $d_G(x,y)$ is the length of a shortest path between $u$ and $v$ and $\mathbb{N}$ is the set of positive integers. Then $(V,d_G)$ is a metric space since $d_G$ satisfies $(i)$ $d_G(x,x)=0$  for all $x\in V$,$(ii)$  $d_G(x,y)=d_G(y,x)$  for all $x,y \in V$ and $(iii)$ $d_G(x,y)\le d_G(x,z)+d_G(z,y)$  for all $x,y,z\in V$. A vertex $v\in V$ is said to \textit{distinguish} two vertices $x$ and $y$ if $d_G(v,x)\ne d_G(v,y)$.
A set $S\subset V$ is said to be a \emph{metric generator} for $G$ if any pair of vertices of $G$ is
distinguished by some element of $S$. A minimum cardinality metric generator is called a \emph{metric basis}, and
its cardinality the \emph{metric dimension} of $G$, denoted by $\dim(G)$.  

The notion of metric dimension of a graph was introduced by Slater in \cite{Slater1975}, where metric generators were called \emph{locating sets}. Harary and Melter independently introduced the same concept in  \cite{Harary1976}, where metric generators were called \emph{resolving sets}. Applications of this invariant to the navigation of robots in networks are discussed in \cite{Khuller1996} and applications to chemistry in \cite{Johnson1993,Johnson1998}.  Several variations of metric generators, including resolving dominating sets \cite{Brigham2003}, independent resolving sets \cite{Chartrand2003}, local metric sets \cite{Okamoto2010}, strong resolving sets \cite{Sebo2004}, adjacency resolving sets  \cite{JanOmo2012}, $k$-metric generators \cite{Estrada-Moreno2013,Estrada-Moreno2013corona}, simultaneous metric generators \cite{Ramirez-Cruz-Rodriguez-Velazquez_2014,Ramirez2014}, etc., have since been introduced and studied. 

The concept of adjacency generator\footnote{Adjacency generators were called adjacency resolving sets in   \cite{JanOmo2012}.} was introduced by Jannesari and Omoomi in \cite{JanOmo2012} as a tool to study the metric dimension of lexicographic product graphs. This concept has been studied further by Fernau and Rodr\'{i}guez-Vel\'{a}zquez in \cite{Rodriguez-Velazquez-Fernau2013,Fernau-Ja-Corona-2014} where they showed
that the (local) metric dimension of the corona product of a graph of order $n$ and some
non-trivial graph $H$ equals $n$ times the (local) adjacency metric dimension of $H$. As a consequence of this strong relation they showed that the problem of computing the adjacency metric dimension is
NP-hard.   A set $S\subset V$ of vertices in a graph $G=(V,E)$ is said to be  an \emph{adjacency generator} for $G$  if for every two vertices $x,y\in V-S$ there exists $s\in S$ such that $s$ is adjacent to exactly one of $x$ and $y$. A minimum cardinality adjacency generator is called an \emph{adjacency basis} of $G$, and its cardinality  the \emph{adjacency dimension} of $G$,  denoted by $\dim_A(G)$.
Since any adjacency basis is a metric generator, $\dim(G)\le \dim_A(G)$.  Besides, for any connected graph $G$ of diameter at most two, $\dim_A(G)=\dim(G)$.

As pointed out in \cite{Rodriguez-Velazquez-Fernau2013,Fernau-Ja-Corona-2014}, 
any  adjacency generator of a graph $G=(V,E)$ is  also a metric generator in a suitably chosen metric space.
Given a positive integer $t$,  we define the distance function $d_{G,t}:V\times V\rightarrow \mathbb{N}\cup \{0\}$, where
\begin{equation*}\label{distinguishAdj}
d_{G,t}(x,y)=\min\{d_G(x,y),t\}.
\end{equation*}
Then any metric generator for $(V,d_{G,t})$ is a metric generator for $(V,d_{G,t+1})$ and, as a consequence, the metric dimension of $(V,d_{G,t+1})$  is less than or equal to the metric dimension of $(V,d_{G,t})$. In particular, the metric dimension of $(V,d_{G,1})$ is equal to $|V|-1$,  the metric dimension of $(V,d_{G,2})$ is equal to $\dim_A(G)$ and, if $G$ has diameter $D(G)$, then $d_{G,D(G)}=d_G$ and so  the metric dimension of  $(V,d_{G,D(G)})$  is equal to $\dim(G)$.
Notice that when using the metric $d_{G,t}$ 
the concept of metric generator needs not be restricted to the case of connected graphs\footnote{For any pair of vertices $x,y$ belonging to different connected components of $G$ we can assume that $d_G(x,y)=\infty$ and so $d_{G,t}(x,y)=t$ for any $t$ greater than or equal to the maximum diameter of a connected component of $G$.}. Notice that $S$ is an adjacency generator for $G$ if and only if $S$ is an adjacency generator for its complement $\overline{G}$. This is justified by the fact that given an adjacency generator $S$ for $G$, it holds that for every $x,y\in V- S$ there exists $s\in S$ such that $s$ is adjacent to exactly one of $x$ and $y$, and this property holds in $\overline{G}$. Thus, $\dim_A(G)=\dim_A(\overline{G}).$

Let ${\mathcal G}=\{G_1,G_2,...,G_k\}$ be a family  of (not necessarily edge-disjoint) connected graphs $G_i=(V,E_i)$ with common vertex set $V$ (the union of whose edge sets is not necessarily the complete graph). Ram\'{i}rez-Cruz, Oellermann  and  Rodr\'{i}guez-Vel\'{a}zquez defined in \cite{Ramirez-Cruz-Rodriguez-Velazquez_2014,Ramirez2014}   a \textit{simultaneous metric generator} for ${\mathcal{G}}$ as a set $S\subset V$ such that $S$ is simultaneously a metric generator for each $G_i$. They say that a minimum simultaneous metric generator for ${\mathcal{G}}$ is  a \emph{simultaneous metric basis} of ${\mathcal{G}}$, and
its cardinality the \emph{simultaneous metric dimension} of ${\mathcal{G}}$, denoted by $\Sd({\mathcal{G}})$ or explicitly by $\Sd( G_1,G_2,...,G_k )$. Analogously, we define a \emph{simultaneous adjacency generator} for ${\cal G}$ to be a set $S\subset V$ such that $S$ is simultaneously an adjacency generator for each $G_i$. We say that a minimum simultaneous adjacency generator for ${\cal G}$ is a \emph{simultaneous adjacency basis} of ${\cal G}$, and
its cardinality the \emph{simultaneous adjacency dimension} of ${\cal G}$, denoted by $\Sd_A({\cal G})$ or explicitly by $\Sd_A(G_1,G_2,...,G_t)$. For instance, the set $\{v_1,v_3,v_5,v_8\}$ is a simultaneous adjacency basis of the family ${\cal G}=\{G_1,G_2,G_3\}$ shown in Figure  \ref{ExSimultaneousAdjBasis},  while the set $\{v_1,v_5,v_8\}$ is a simultaneous metric basis, so $\Sd_A({\cal G})=4>3=\Sd({\cal G})$.

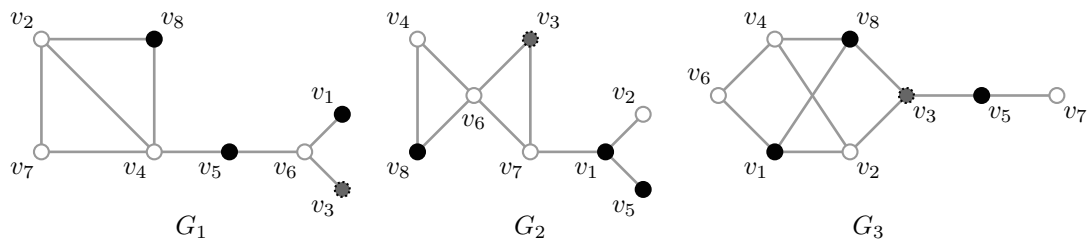
\begin{figure}[h]
\begin{center}
\begin{tikzpicture}[inner sep=0.7mm, place/.style={circle,draw=black!40,
fill=white,thick},dxx/.style={circle,draw=black!99,fill=black!99,thick},lxx/.style={circle,draw=black!99,densely dotted,fill=black!60,thick},transition/.style={rectangle,draw=black!50,fill=black!20,thick},line width=1pt,scale=0.5]

\coordinate (A) at (0,1);
\coordinate (B) at (3,1);
\coordinate (C) at (3,4);
\coordinate (D) at (0,4);
\coordinate (E) at (5,1);
\coordinate (F) at (7,1);
\coordinate (G) at (8,0);
\coordinate (H) at (8,2);

\coordinate (I) at (10,1);
\coordinate (J) at (11.5,2.5);
\coordinate (K) at (13,1);
\coordinate (L) at (13,4);
\coordinate (M) at (10,4);
\coordinate (N) at (15,1);
\coordinate (O) at (16,0);
\coordinate (P) at (16,2);

\coordinate (Q) at (18,2.5);
\coordinate (R) at (19.5,1);
\coordinate (S) at (21.5,1);
\coordinate (T) at (23,2.5);
\coordinate (U) at (21.5,4);
\coordinate (V) at (19.5,4);
\coordinate (W) at (25,2.5);
\coordinate (X) at (27,2.5);

\draw[black!40] (B) -- (C) -- (D) -- (A) -- (B) -- (E) -- (F) -- (G);
\draw[black!40] (B) -- (D);
\draw[black!40] (F) -- (H);

\draw[black!40] (K) -- (L) -- (J) -- (M) -- (I) -- (J) -- (K) -- (N) -- (O);
\draw[black!40] (N) -- (P);

\draw[black!40] (X) -- (W) -- (T) -- (S) -- (R) -- (Q) -- (V) -- (U) -- (T);
\draw[black!40] (V) -- (S);
\draw[black!40] (U) -- (R);

\node at (A) [place]  {};
\coordinate [label=center:{$v_7$}] (v17) at (-0.5,0.5);
\node at (B) [place]  {};
\coordinate [label=center:{$v_4$}] (v14) at (2.5,0.5);
\node at (C) [dxx]  {};
\coordinate [label=center:{$v_8$}] (v18) at (3.5,4.5);
\node at (D) [place]  {};
\coordinate [label=center:{$v_2$}] (v12) at (-0.5,4.5);
\node at (E) [dxx]  {};
\coordinate [label=center:{$v_5$}] (v15) at (4.5,0.5);
\node at (F) [place]  {};
\coordinate [label=center:{$v_6$}] (v16) at (6.5,0.5);
\node at (G) [lxx]  {};
\coordinate [label=center:{$v_3$}] (v13) at (7.5,-0.5);
\node at (H) [dxx]  {};
\coordinate [label=center:{$v_1$}] (v11) at (7.5,2.5);

\node at (I) [dxx]  {};
\coordinate [label=center:{$v_8$}] (v28) at (9.5,0.5);
\node at (J) [place]  {};
\coordinate [label=center:{$v_6$}] (v26) at (11.5,1.8);
\node at (K) [place]  {};
\coordinate [label=center:{$v_7$}] (v27) at (12.5,0.5);
\node at (L) [lxx]  {};
\coordinate [label=center:{$v_3$}] (v23) at (13.5,4.5);
\node at (M) [place]  {};
\coordinate [label=center:{$v_4$}] (v24) at (9.5,4.5);
\node at (N) [dxx]  {};
\coordinate [label=center:{$v_1$}] (v21) at (14.5,0.5);
\node at (O) [dxx]  {};
\coordinate [label=center:{$v_5$}] (v25) at (15.5,-0.5);
\node at (P) [place]  {};
\coordinate [label=center:{$v_2$}] (v22) at (15.5,2.5);

\node at (Q) [place]  {};
\coordinate [label=center:{$v_6$}] (v32) at (17.5,3);
\node at (R) [dxx]  {};
\coordinate [label=center:{$v_1$}] (v38) at (19,0.5);
\node at (S) [place]  {};
\coordinate [label=center:{$v_2$}] (v33) at (22,0.5);
\node at (T) [lxx]  {};
\coordinate [label=center:{$v_3$}] (v36) at (23.5,2);
\node at (U) [dxx]  {};
\coordinate [label=center:{$v_8$}] (v34) at (22,4.5);
\node at (V) [place]  {};
\coordinate [label=center:{$v_4$}] (v31) at (19,4.5);
\node at (W) [dxx]  {};
\coordinate [label=center:{$v_5$}] (v35) at (25.5,2);
\node at (X) [place]  {};
\coordinate [label=center:{$v_7$}] (v37) at (27.5,2);

\coordinate [label=center:{$G_1$}] (G1) at (4,-1);
\coordinate [label=center:{$G_2$}] (G2) at (13,-1);
\coordinate [label=center:{$G_3$}] (G3) at (22,-1);

\end{tikzpicture}

\end{center}
\caption{The set $\{v_1,v_3,v_5,v_8\}$ is a simultaneous adjacency basis of ${\cal G}=\{G_1,G_2,G_3\}$, whereas $\{v_1,v_5,v_8\}$ is a simultaneous metric basis.}
\label{ExSimultaneousAdjBasis}
\end{figure}

The study of simultaneous parameters in graphs was introduced by Brigham and  Dutton in \cite{Brigham1990}, where they studied simultaneous  domination. 
This should not be  confused with studies on families sharing a constant value on a parameter, for instance the study presented in \cite{IMRAN2013}, where several graph families such that all members have a constant metric dimension are studied, enforcing no constraints regarding whether all members share a metric basis or not. In particular, the study of the simultaneous metric dimension  was introduced  in \cite{Ramirez-Cruz-Rodriguez-Velazquez_2014,Ramirez2014}, where the authors obtained sharp bounds for this invariant  for general families of graphs and gave closed formulae
or tight bounds for the simultaneous metric dimension of several specific graph families. For a given graph $G$ they described a process for obtaining a lower bound
on the maximum number of graphs in a family containing $G$ that has simultaneous metric dimension equal to $\dim(G)$. Moreover, it was shown that the problem of finding the simultaneous metric dimension of families of trees is $NP$-hard.

In this paper, we introduce the simultaneous adjacency dimension and study several of its properties, as well as its relations to the simultaneous metric dimension. We put a special focus on the usefulness of the simultaneous adjacency dimension as a tool for the study of the simultaneous metric dimension of families composed by lexicographic product graphs. Let $G$ be a graph of order $n$, and let $(H_1,H_2,\ldots,H_n)$ be an ordered $n$-tuple of graphs of order $n'_1$, $n'_2$, \ldots, $n'_n$, respectively. The \emph{lexicographic product} of $G$ and $(H_1,H_2,\ldots,H_n)$ is the graph $G \circ (H_1,H_2,\ldots,H_n)$, such that $V(G \circ (H_1,H_2,\ldots,H_n))=\bigcup_{u_i \in V(G)} (\{u_i\} \times V(H_i))$ and $(u_i,v_r)(u_j,v_s) \in E(G \circ (H_1,H_2,\ldots,H_n))$ if and only if $u_iu_j \in E(G)$ or $i=j$ and $v_rv_s \in E(H_i)$. We will restrict our study to two particular cases. First, given two vertex-disjoint graphs $G=(V_{1},E_{1})$ and $H=(V_{2},E_{2})$, the \emph{join} of $G$ and $H$, denoted as $G+H$, is the graph with vertex set $V(G+H)=V_{1}\cup V_{2}$ and edge set $E(G+H)=E_{1}\cup E_{2}\cup \{uv\,:\,u\in V_{1},v\in V_{2}\}$. Join graphs are lexicographic product graphs, as $G+H \cong P_2 \circ (G,H)$. The other particular case we will focus on is the traditional lexicographic product graph, where $H_i \cong H$ for every $i \in \{1,\ldots,n\}$, which is denoted as $G \circ H$ for simplicity.

The remainder of the paper is structured as follows. Section~\ref{sectBasicResults} covers general bounds and basic results on the simultaneous adjacency dimension. Section~\ref{sectFamJoins} is devoted to families composed by join graphs, whereas Section~\ref{sectFamStdLex} covers families composed by standard lexicographic product graphs. Throughout the paper, we will use the notation $K_n$, $K_{r,s}$, $C_n$, $N_n$ and $P_n$ for complete graphs, complete bipartite graphs, cycle graphs, empty graphs and path graphs of order $n$, respectively. We use the notation $u \sim v$ if $u$ and $v$ are adjacent and $G \cong H$ if $G$ and $H$ are isomorphic graphs. For a vertex $v$ of a graph $G$, $N_G(v)$ will denote the set of neighbours or \emph{open neighbourhood} of $v$ in $G$, \emph{i.e.} $N_G(v)=\{u \in V(G):\; u \sim v\}$. The \emph{closed neighbourhood}, denoted by $N_G[v]$, equals $N_G(v) \cup \{v\}$. If there is no ambiguity, we will simple write $N(v)$ or $N[v]$. Two vertices $x,y\in V(G)$  are \textit{twins} in $G$ if $N_G[x]=N_G[y]$ or $N_G(x)=N_G(y)$. If $N_G[x]=N_G[y]$, they are said to be \emph{true twins}, whereas if $N_G(x)=N_G(y)$ they are said to be \emph{false twins}. We denote by $\delta(v)=|N(v)|$ the degree of vertex $v$, as well as $\delta(G)=\min_{v \in V(G)}\{\delta(v)\}$ and $\Delta(G)=\max_{v \in V(G)}\{\delta(v)\}$. The subgraph of $G$ induced by a set $S$ of vertices is denoted  by $\langle S\rangle_G$. If there is no ambiguity, we will simply write $\langle S\rangle$. The diameter of a  graph $G$ is denoted by $D(G)$ and its girth by $\mathtt{g}(G)$. For the remainder of the paper, definitions will be introduced whenever a concept is needed.

\section{Basic results on the simultaneous adjacency dimension}\label{sectBasicResults}

\begin{remark}\label{trivialBoundsSimAdjDim}
For any family ${\cal G}=\{G_1,G_2,...,G_t\}$ of  connected graphs on a common vertex set $V$, the following results hold:
\begin{enumerate}[{\rm (i)}]
\item $\Sd_A({\cal G}) \geq \underset{i\in \{1,...,k\}}{\max}\{\dim_A(G_i)\}$.
\item $\Sd_A({\cal G})\geq\Sd({\cal G})$.
\item $\Sd_A({\cal G}) \leq \vert V\vert-1$.
\end{enumerate}
\end{remark}

\begin{proof}
(i) is deduced directly from the definition of simultaneous adjacency dimension, while (iii) is obtained from the fact that for any non-trivial graph $G=(V,E)$ it holds that for any $v\in V$ the set $V-\{v\}$ is an adjacency generator. Let $B$ be a simultaneous adjacency basis of ${\cal G}$ and let $u,v \in V-B$, be two different vertices. For every graph $G_i$, there exists $x \in B$ such that $d_{G_i,2}(u,x)\ne d_{G_i,2}(v,x)$, so $d_{G_i}(u,x) \neq d_{G_i}(v,x)$. Thus, $B$ is a simultaneous metric generator for ${\cal G}$ and, as a consequence, (ii) follows.
\end{proof}

As pointed out in \cite{JanOmo2012},  $\dim_A(G)=n-1$ if and only if $G=K_n$ or $G=N_n$. The following result follows directly from Remark \ref{trivialBoundsSimAdjDim}.

\begin{corollary}\label{casesV_1}
Let ${\cal G}$ be a graph family on a common vertex set $V$. If $K_{\vert V \vert} \in {\cal G}$ or $N_{\vert V \vert} \in {\cal G}$, then $\Sd_A({\cal G})=\vert V \vert - 1$.
\end{corollary}

The converse of Corollary~\ref{casesV_1} does not hold, as we will exemplify in Corollary~\ref{partCaseStars}. We first recall a characterization of the cases where $\Sd({\cal G})=|V|-1$, presented in \cite{Ramirez-Cruz-Rodriguez-Velazquez_2014}.

\begin{theorem}\label{SdVminus1}{\rm \cite{Ramirez-Cruz-Rodriguez-Velazquez_2014}}
Let ${\mathcal G}$ be a  family of connected graphs on a common vertex set $V$. Then $\Sd({\mathcal G})=\vert V\vert-1$ if and only if for every pair $u,v\in V$, there exists a graph $G_{uv}\in {\mathcal G}$ such that $u$ and $v$ are twins in $G_{uv}$.
\end{theorem}

The following result is a direct consequence of Remark~\ref{trivialBoundsSimAdjDim} (ii), (iii) and Theorem~\ref{SdVminus1}.

\begin{remark}\label{SdAVminus1Twins}
Let ${\mathcal G}$ be a graph family on a common vertex set $V$. If for every pair $u,v\in V$, there exists a graph $G_{uv}\in {\mathcal G}$ such that $u$ and $v$ are twins in $G_{uv}$, then $\Sd_A({\cal G})=|V|-1$.
\end{remark}

For a star graph $K_{1,r}$, $r \ge 3$, it is known that $\dim_A(K_{1,r})=r-1$ and every adjacency basis is composed by all but one of its leaves. For a finite set $V=\{v_1,v_2,\ldots,v_n\}$, $n\ge 4$, let $K_{1,n-1}^i$ be the star graph having $v_i$ as its central vertex and $V-\{v_i\}$ as its leaves. We define the family ${\cal K}(V)=\{K_{1,n-1}^i :\; v_i \in V\}$. Any pair of vertices $v_p,v_q \in V$ are twins in every $K_{1,n-1}^i \in {\cal K}(V)-\{K_{1,n-1}^p,K_{1,n-1}^q\}$, so the following result is a direct consequence of Remark~\ref{SdAVminus1Twins}.

\begin{corollary}\label{partCaseStars}
For every finite set $V$ of size $|V|\ge 4$, $\Sd_A({\cal K}(V))=|V|-1$.
\end{corollary}

Let $P_3^{(1)}=(V,E_1)$, $P_3^{(2)}=(V,E_2)$ and $P_3^{(3)}=(V,E_3)$ be the three different path graphs  defined on the common vertex set $V=\{v_1,v_2,v_3\}$, where $v_i$ is the vertex of degree two in $P_3^{(i)}$, for $i\in \{1,2,3\}$. It was shown in \cite{JanOmo2012} that $\dim_A(G)=1$ if and only if $G \in \{P_1,P_2,P_3,\overline{P}_2,\overline{P}_3\}$. The following result follows directly from this fact.

\begin{remark}
The following statements hold:
\begin{enumerate}[{\rm (i)}]
\item $\Sd_A({\cal G})=1$ if and only if ${\cal G} \subseteq \{P_2,\overline{P}_2\}$, ${\cal G} \subseteq \{P_3^{(1)},P_3^{(2)},\overline{P}_3^{(1)},\overline{P}_3^{(2)}\}$, ${\cal G} \subseteq \{P_3^{(1)},P_3^{(3)},\overline{P}_3^{(1)},$ $\overline{P}_3^{(3)}\}$ or ${\cal G} \subseteq \{P_3^{(2)},P_3^{(3)},\overline{P}_3^{(2)},\overline{P}_3^{(3)}\}$.
\item $\Sd_A(P_3^{(1)}, P_3^{(2)},P_3^{(3)},\overline{P}_3^{(1)},\overline{P}_3^{(2)},\overline{P}_3^{(3)})=2$.
\end{enumerate}
\end{remark}

 The following result  is derived from the fact that any graph and its complement have the same adjacency bases.

\begin{remark}\label{simAdjDimCompl}
Let ${\cal G}=\{G_1,G_2,\ldots,G_k\}$ be a family of graphs with the same vertex set $V$, and let $\overline{\cal G}=\{\overline{G}_1,\overline{G}_2,\ldots,\overline{G}_k\}$ be the family composed by the complements of every graph in ${\cal G}$. The following assertions hold:
\begin{enumerate}[{\rm (i)}]
\item $\Sd_A({\cal G})=\Sd_A(\overline{\cal G})=\Sd_A({\cal G}\cup \overline{\cal G})$. Moreover, the simultaneous adjacency bases of ${\cal G}$ and $\overline{\cal G}$ coincide.
\item For any subfamily of graphs ${\cal G}' \subseteq \overline{\cal G}$, $\Sd_A({\cal G})=\Sd_A({\cal G} \cup {\cal G}')$.
\end{enumerate}
\end{remark}

Let $G=(V,E)$ be a graph and let  $\Perm(V)$ be the set of all permutations of $V$. Given a subset $X\subseteq V$, the stabilizer of $X$ is the set of permutations $${\cal S}(X)=\{f\in \Perm(V): f(x)=x, \; \mbox{\rm for every } x\in X\} .$$  As usual, we denote by $f(X)$ the image of a subset $X$ under $f$, \textit{i.e}., $f(X)=\{f(x):\; x\in X\}$.

Let $G=(V,E)$ be a graph and let $B\subset V$ be a nonempty set. For any  
permutation $f\in {\cal S}(B)$ of $V$ we say that a graph $G'=(V,E')$ belongs to the  family  ${\cal  G}_{B,f}(G)$ if and only if $N_{G'}(x)=f(N_G(x))$, for every $x\in B$. We define the subgraph $\langle B_G\rangle_{w}=(N_G[B],E_{w})$ of $G$, weakly induced by $B$,  where $N_G[B]=\cup_{x\in B}N_G[x]$ and $E_{w}$ is the set of all edges having at least one vertex
in $B$.


 \begin{remark}\label{Remark_induced_subgraphs_inG'}
 Let $G=(V,E)$ be a graph and let $B\subset V$ be a nonempty set. For any  $f\in {\cal{S}}(B)$ and any graph  $G'\in  {\cal{G}}_{B,f}$, $$\langle B_G \rangle_w \cong \langle B_{G'} \rangle_w.$$
 \end{remark}

\begin{proof}
Since $G'\in {\cal G}_{B,f}$, the function $f$ is a bijection from $V(G)$ onto $V(G')$. Now, since $N_{G'}(x)=f(N_G(x))$, for every $x\in B$, we conclude that $uv$ is an edge of $\langle B_G \rangle_w$ if and only if $f(u)f(v)$ is an edge of $\langle B_{G'} \rangle_w$. Therefore, 
 the restriction of $f$ to $\langle B_G \rangle_w$ is an isomorphism.
 \end{proof}

Now we define the family of graphs ${\cal  G}_B(G)$, associated to $B$, as follows.
$${\cal  G}_B(G)=\displaystyle\bigcup_{f\in {\cal S}(B)}{\cal  G}_{B,f}(G).$$

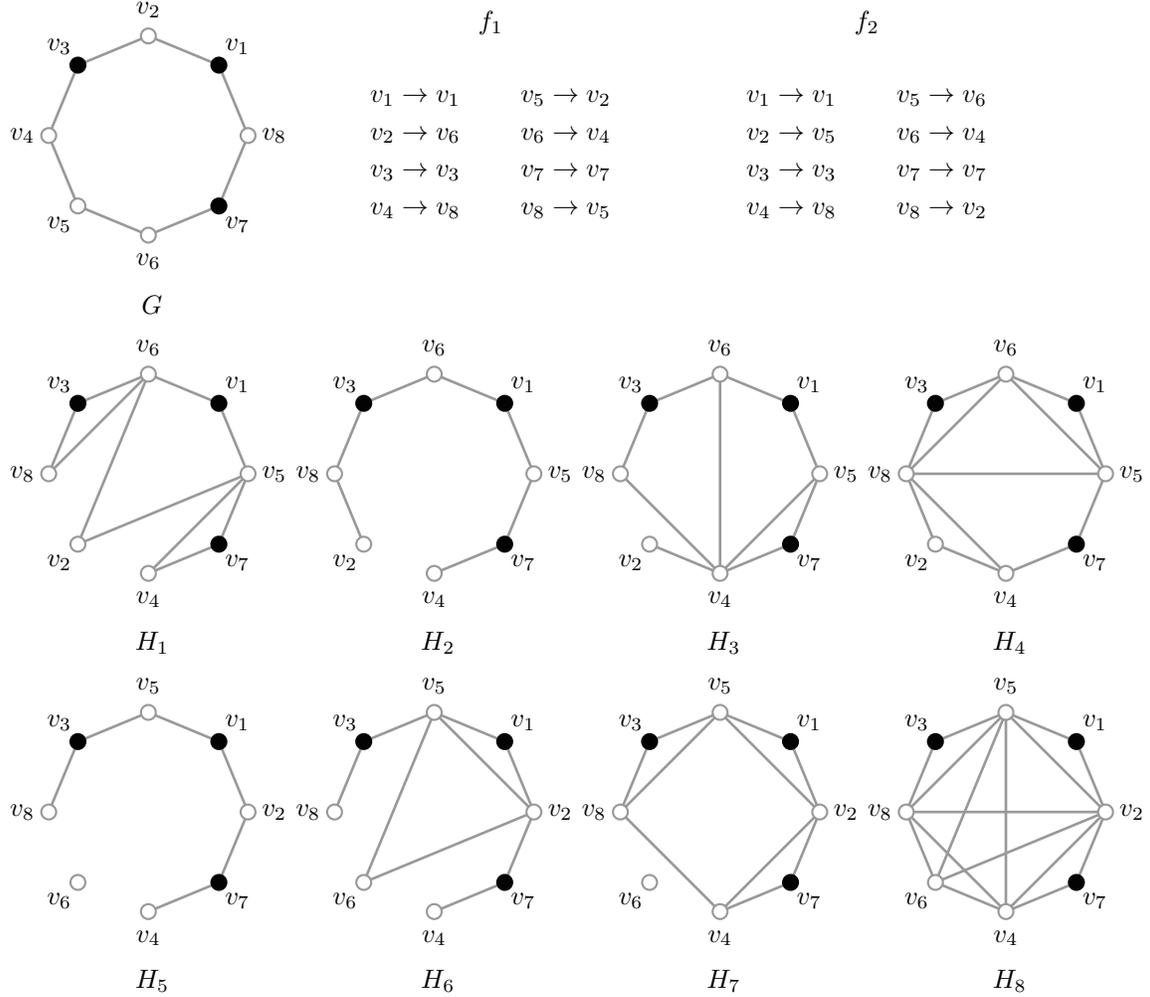
\begin{figure}[h]
\begin{center}
\begin{tikzpicture}[inner sep=0.7mm, place/.style={circle,draw=black!40,
fill=white,thick},xx/.style={circle,draw=black!99,fill=black!99,thick},
transition/.style={rectangle,draw=black!50,fill=black!20,thick},line width=1pt,scale=0.5]

\def\n{8}

\def\radius{2.65cm}
\def\lblradius{3.35cm}


\foreach \ind in {1,...,\n}\pgfmathparse{360/\n*\ind}\coordinate (b\ind) at (\pgfmathresult:\radius);

\foreach \ind in {1,...,\n}\pgfmathparse{360/\n*\ind}\coordinate (v0\ind) at (\pgfmathresult:\lblradius);

\draw[black!40] (b1) -- (b2) -- (b3) -- (b4) -- (b5) -- (b6) -- (b7) -- (b8) -- cycle;

\node [xx] at (b1) {};
\node at (v01) {$v_1$};
\node [place] at (b2) {};
\node at (v02) {$v_2$};
\node [xx] at (b3) {};
\node at (v03) {$v_3$};
\node [place] at (b4) {};
\node at (v04) {$v_4$};
\node [place] at (b5) {};
\node at (v05) {$v_5$};
\node [place] at (b6) {};
\node at (v06) {$v_6$};
\node [xx] at (b7) {};
\node at (v07) {$v_7$};
\node [place] at (b8) {};
\node at (v08) {$v_8$};

\coordinate [label=center:{$G$}] (G) at (0.1,-4.5);


\coordinate [label=center:{$f_1$}] (f1) at (9.1,3);

\coordinate [label=center:{$v_1 \rightarrow v_1$}] (f111) at (7.1,1);
\coordinate [label=center:{$v_2 \rightarrow v_6$}] (f126) at (7.1,0);
\coordinate [label=center:{$v_3 \rightarrow v_3$}] (f133) at (7.1,-1);
\coordinate [label=center:{$v_4 \rightarrow v_8$}] (f148) at (7.1,-2);

\coordinate [label=center:{$v_5 \rightarrow v_2$}] (f152) at (11.1,1);
\coordinate [label=center:{$v_6 \rightarrow v_4$}] (f164) at (11.1,0);
\coordinate [label=center:{$v_7 \rightarrow v_7$}] (f177) at (11.1,-1);
\coordinate [label=center:{$v_8 \rightarrow v_5$}] (f185) at (11.1,-2);


\coordinate [label=center:{$f_2$}] (f2) at (19.1,3);

\coordinate [label=center:{$v_1 \rightarrow v_1$}] (f211) at (17.1,1);
\coordinate [label=center:{$v_2 \rightarrow v_5$}] (f225) at (17.1,0);
\coordinate [label=center:{$v_3 \rightarrow v_3$}] (f233) at (17.1,-1);
\coordinate [label=center:{$v_4 \rightarrow v_8$}] (f248) at (17.1,-2);

\coordinate [label=center:{$v_5 \rightarrow v_6$}] (f256) at (21.1,1);
\coordinate [label=center:{$v_6 \rightarrow v_4$}] (f264) at (21.1,0);
\coordinate [label=center:{$v_7 \rightarrow v_7$}] (f277) at (21.1,-1);
\coordinate [label=center:{$v_8 \rightarrow v_2$}] (f282) at (21.1,-2);


\foreach \ind in {1,...,\n}\pgfmathparse{360/\n*\ind}\coordinate [yshift=-4.5cm] (g1\ind) at (\pgfmathresult:\radius);

\foreach \ind in {1,...,\n}\pgfmathparse{360/\n*\ind}\coordinate [yshift=-4.5cm] (v1\ind) at (\pgfmathresult:\lblradius);

\draw[black!40] (g16) -- (g17) -- (g18) -- (g11) -- (g12) -- (g13) -- (g14) -- (g12) -- (g15) -- (g18) -- (g16);

\node [xx] at (g11) {};
\node at (v11) {$v_1$};
\node [place] at (g12) {};
\node at (v12) {$v_6$};
\node [xx] at (g13) {};
\node at (v13) {$v_3$};
\node [place] at (g14) {};
\node at (v14) {$v_8$};
\node [place] at (g15) {};
\node at (v15) {$v_2$};
\node [place] at (g16) {};
\node at (v16) {$v_4$};
\node [xx] at (g17) {};
\node at (v17) {$v_7$};
\node [place] at (g18) {};
\node at (v18) {$v_5$};

\coordinate [label=center:{$H_1$}] (H1) at (0.1,-13.5);


\foreach \ind in {1,...,\n}\pgfmathparse{360/\n*\ind}\coordinate [xshift=3.8cm,yshift=-4.5cm] (g2\ind) at (\pgfmathresult:\radius);

\foreach \ind in {1,...,\n}\pgfmathparse{360/\n*\ind}\coordinate [xshift=3.8cm,yshift=-4.5cm] (v2\ind) at (\pgfmathresult:\lblradius);

\draw[black!40] (g26) -- (g27) -- (g28) -- (g21) -- (g22) -- (g23) -- (g24) -- (g25);

\node [xx] at (g21) {};
\node at (v21) {$v_1$};
\node [place] at (g22) {};
\node at (v22) {$v_6$};
\node [xx] at (g23) {};
\node at (v23) {$v_3$};
\node [place] at (g24) {};
\node at (v24) {$v_8$};
\node [place] at (g25) {};
\node at (v25) {$v_2$};
\node [place] at (g26) {};
\node at (v26) {$v_4$};
\node [xx] at (g27) {};
\node at (v27) {$v_7$};
\node [place] at (g28) {};
\node at (v28) {$v_5$};

\coordinate [label=center:{$H_2$}] (H2) at (7.7,-13.5);


\foreach \ind in {1,...,\n}\pgfmathparse{360/\n*\ind}\coordinate [xshift=7.6cm,yshift=-4.5cm] (g3\ind) at (\pgfmathresult:\radius);

\foreach \ind in {1,...,\n}\pgfmathparse{360/\n*\ind}\coordinate [xshift=7.6cm,yshift=-4.5cm] (v3\ind) at (\pgfmathresult:\lblradius);

\draw[black!40] (g35) -- (g36) -- (g37) -- (g38) -- (g31) -- (g32) -- (g33) -- (g34);
\draw[black!40] (g36) -- (g34);
\draw[black!40] (g36) -- (g32);
\draw[black!40] (g36) -- (g38);

\node [xx] at (g31) {};
\node at (v31) {$v_1$};
\node [place] at (g32) {};
\node at (v32) {$v_6$};
\node [xx] at (g33) {};
\node at (v33) {$v_3$};
\node [place] at (g34) {};
\node at (v34) {$v_8$};
\node [place] at (g35) {};
\node at (v35) {$v_2$};
\node [place] at (g36) {};
\node at (v36) {$v_4$};
\node [xx] at (g37) {};
\node at (v37) {$v_7$};
\node [place] at (g38) {};
\node at (v38) {$v_5$};

\coordinate [label=center:{$H_3$}] (H3) at (15.3,-13.5);


\foreach \ind in {1,...,\n}\pgfmathparse{360/\n*\ind}\coordinate [xshift=11.4cm,yshift=-4.5cm] (g4\ind) at (\pgfmathresult:\radius);

\foreach \ind in {1,...,\n}\pgfmathparse{360/\n*\ind}\coordinate [xshift=11.4cm,yshift=-4.5cm] (v4\ind) at (\pgfmathresult:\lblradius);

\draw[black!40] (g41) -- (g42) -- (g43) -- (g44) -- (g45) -- (g46) -- (g47) -- (g48) -- cycle;
\draw[black!40] (g46) -- (g44);
\draw[black!40] (g44) -- (g48);
\draw[black!40] (g42) -- (g44);
\draw[black!40] (g42) -- (g48);

\node [xx] at (g41) {};
\node at (v41) {$v_1$};
\node [place] at (g42) {};
\node at (v42) {$v_6$};
\node [xx] at (g43) {};
\node at (v43) {$v_3$};
\node [place] at (g44) {};
\node at (v44) {$v_8$};
\node [place] at (g45) {};
\node at (v45) {$v_2$};
\node [place] at (g46) {};
\node at (v46) {$v_4$};
\node [xx] at (g47) {};
\node at (v47) {$v_7$};
\node [place] at (g48) {};
\node at (v48) {$v_5$};

\coordinate [label=center:{$H_4$}] (H4) at (22.9,-13.5);


\foreach \ind in {1,...,\n}\pgfmathparse{360/\n*\ind}\coordinate [yshift=-9cm] (g5\ind) at (\pgfmathresult:\radius);

\foreach \ind in {1,...,\n}\pgfmathparse{360/\n*\ind}\coordinate [yshift=-9cm] (v5\ind) at (\pgfmathresult:\lblradius);

\draw[black!40] (g56) -- (g57) -- (g58) -- (g51) -- (g52) -- (g53) -- (g54);

\node [xx] at (g51) {};
\node at (v51) {$v_1$};
\node [place] at (g52) {};
\node at (v52) {$v_5$};
\node [xx] at (g53) {};
\node at (v53) {$v_3$};
\node [place] at (g54) {};
\node at (v54) {$v_8$};
\node [place] at (g55) {};
\node at (v55) {$v_6$};
\node [place] at (g56) {};
\node at (v56) {$v_4$};
\node [xx] at (g57) {};
\node at (v57) {$v_7$};
\node [place] at (g58) {};
\node at (v58) {$v_2$};

\coordinate [label=center:{$H_5$}] (H5) at (0.1,-22.5);


\foreach \ind in {1,...,\n}\pgfmathparse{360/\n*\ind}\coordinate [xshift=3.8cm,yshift=-9cm] (g6\ind) at (\pgfmathresult:\radius);

\foreach \ind in {1,...,\n}\pgfmathparse{360/\n*\ind}\coordinate [xshift=3.8cm,yshift=-9cm] (v6\ind) at (\pgfmathresult:\lblradius);

\draw[black!40] (g66) -- (g67) -- (g68) -- (g61) -- (g62) -- (g63) -- (g64);
\draw[black!40] (g62) -- (g65) -- (g68) -- cycle;

\node [xx] at (g61) {};
\node at (v61) {$v_1$};
\node [place] at (g62) {};
\node at (v62) {$v_5$};
\node [xx] at (g63) {};
\node at (v63) {$v_3$};
\node [place] at (g64) {};
\node at (v64) {$v_8$};
\node [place] at (g65) {};
\node at (v65) {$v_6$};
\node [place] at (g66) {};
\node at (v66) {$v_4$};
\node [xx] at (g67) {};
\node at (v67) {$v_7$};
\node [place] at (g68) {};
\node at (v68) {$v_2$};

\coordinate [label=center:{$H_6$}] (H6) at (7.7,-22.5);


\foreach \ind in {1,...,\n}\pgfmathparse{360/\n*\ind}\coordinate [xshift=7.6cm,yshift=-9cm] (g7\ind) at (\pgfmathresult:\radius);

\foreach \ind in {1,...,\n}\pgfmathparse{360/\n*\ind}\coordinate [xshift=7.6cm,yshift=-9cm] (v7\ind) at (\pgfmathresult:\lblradius);

\draw[black!40] (g71) -- (g72) -- (g73) -- (g74) -- (g76) -- (g77) -- (g78) -- cycle;
\draw[black!40] (g74) -- (g72) -- (g78) -- (g76);

\node [xx] at (g71) {};
\node at (v71) {$v_1$};
\node [place] at (g72) {};
\node at (v72) {$v_5$};
\node [xx] at (g73) {};
\node at (v73) {$v_3$};
\node [place] at (g74) {};
\node at (v74) {$v_8$};
\node [place] at (g75) {};
\node at (v75) {$v_6$};
\node [place] at (g76) {};
\node at (v76) {$v_4$};
\node [xx] at (g77) {};
\node at (v77) {$v_7$};
\node [place] at (g78) {};
\node at (v78) {$v_2$};

\coordinate [label=center:{$H_7$}] (H7) at (15.3,-22.5);


\foreach \ind in {1,...,\n}\pgfmathparse{360/\n*\ind}\coordinate [xshift=11.4cm,yshift=-9cm] (g8\ind) at (\pgfmathresult:\radius);

\foreach \ind in {1,...,\n}\pgfmathparse{360/\n*\ind}\coordinate [xshift=11.4cm,yshift=-9cm] (v8\ind) at (\pgfmathresult:\lblradius);

\draw[black!40] (g81) -- (g82) -- (g83) -- (g84)-- (g85) -- (g86) -- (g87) -- (g88) -- cycle;
\draw[black!40] (g82) -- (g84) -- (g86) -- (g88) -- (g82) -- (g85) -- (g88);
\draw[black!40] (g88) -- (g84);
\draw[black!40] (g82) -- (g86);

\node [xx] at (g81) {};
\node at (v81) {$v_1$};
\node [place] at (g82) {};
\node at (v82) {$v_5$};
\node [xx] at (g83) {};
\node at (v83) {$v_3$};
\node [place] at (g84) {};
\node at (v84) {$v_8$};
\node [place] at (g85) {};
\node at (v85) {$v_6$};
\node [place] at (g86) {};
\node at (v86) {$v_4$};
\node [xx] at (g87) {};
\node at (v87) {$v_7$};
\node [place] at (g88) {};
\node at (v88) {$v_2$};

\coordinate [label=center:{$H_8$}] (H8) at (22.9,-22.5);
\end{tikzpicture}

\end{center}
\caption{A subfamily ${\cal H}$ of ${\cal G}_B(C_8)$ for $B=\{v_1,v_3,v_7\}$, where $\{H_1,H_2,H_3,H_4\} \subseteq {\cal G}_{B,f_1}(C_8)$ and $\{H_5,H_6,H_7,H_8\} \subseteq {\cal G}_{B,f_2}(C_8)$. For every $H_i \in {\cal H}$, $\dim_A(H_i)=\dim_A(C_8)=3$. Moreover, $B$ is a simultaneous adjacency basis of ${\cal H}$, so $\Sd_A({\cal H})=3$.}
\label{FigSubFamilyC8}
\end{figure}

With the above notation in mind we can state our next result. 

\begin{theorem}\label{dimAPermFamily}
Any  adjacency basis $B$    of a graph $G$ is  a simultaneous adjacency generator for  any family of graphs ${\cal  H} \subseteq {\cal  G}_B(G)$. Moreover, if   $G\in {\cal  H}$, then
 $$\Sd_A({\cal  H})=\dim_A(G).$$
\end{theorem}

\begin{proof}
Assume that $B$ is an adjacency basis of a graph $G=(V,E)$.
Let $f\in {\cal S}(B)$ and let $G'=(V,E')$ such that $N_{G'}(x)=f(N_G(x))$, for every $x\in B$.
We will show that $B$ is an adjacency generator for any graph $G'$.  
To this end, we take two different vertices $u',v'\in V-B$ of $G'$ and the corresponding vertices $u,v\in V$ of $G$ such that $f(u)=u'$ and $f(v)=v'$. Since $u\ne v$ and $u,v\not\in B$, there exists $x\in B$ such that $d_{G ,2}(u,x)\ne d_{G ,2}(v,x)$. Now, since $N_{G'}(x)=f(N_G(x))=\{f(w):\; w\in N_G(x)\}$, we obtain that $d_{G',2}(u',x)=d_{G,2}(u,x)\ne d_{G,2}(u,x) = d_{G',2}(v',x)$.
Hence, $B$ is an adjacency generator for $G'$ and, in consequence, is also a simultaneous adjacency generator for ${\cal  H}$. Then we conclude that $\Sd_A({\cal  H})\le |B| =\dim_A(G)$ and, if  $G\in {\cal  H}$, then $\Sd_A({\cal  H})\ge   \dim_A(G)$. Therefore, the result follows.  
\end{proof}

Notice that if $G\not\in \{K_n,N_n\}$, then the edge set of  any graph $G'\in {\cal{G}}_B(G) $  can be partitioned into two sets $E_1$, $E_2$, where $E_1$ consists of all edges of $G$ having at least one vertex in $B$ and $E_2$ is a subset of edges of a complete graph whose vertex set is $V-B$. Hence, ${\cal{G}}_B(G)$ contains
$2^{\frac{|V-B|(|V-B|-1)}{2}}|V-B|!$ different labelled graphs. As an example of large graph families that may be obtained according to this procedure, consider the cycle graph $C_8$, where $\dim_A(C_8)=3$. For each adjacency basis $B$ of $C_8$, we have that $\vert {\cal G}_B(C_8) \vert = 122880$. To illustrate this, Figure~\ref{FigSubFamilyC8} shows a graph family ${\cal H}=\{H_1,\ldots,H_8\} \subseteq {\cal G}_B(C_8)$, where $B=\{v_1,v_3,v_7\}$, $\{H_1,H_2,H_3,H_4\} \subseteq {\cal G}_{B,f_1}(C_8)$ and $\{H_5,H_6,H_7,H_8\} \subseteq {\cal G}_{B,f_2}(C_8)$.

The following result follows directly from Theorem~\ref{dimAPermFamily} and the fact that $\dim_A(G)=1$ if and only if $G \in \{P_2,P_3,\overline{P}_2,\overline{P}_3\}$.

\begin{corollary}\label{CaseDim=2}
Let $G$ be a graph of order $n\ge 4$. If $\dim_A(G)=2$, then for any adjacency basis $B$ of  $G$ and  any non-empty subfamily ${\cal  H} \subseteq {\cal  G}_B(G)$,
 $$\Sd_A({\cal  H})=2.$$
\end{corollary}

The next result, obtained in \cite{Estrada-Moreno2014a}, shows that Corollary \ref{CaseDim=2} is  only applicable  to families of graphs of order $4,5$ or $6$.

\begin{remark}{\rm\cite{Estrada-Moreno2014a}}\label{remAd_1-3}
If $G$ is a graph of order $n\ge 7$, then $\dim_A(G)\ge 3$.
\end{remark}

Theorem \ref{dimAPermFamily} and Remark \ref{remAd_1-3} immediately lead to the next result.

\begin{theorem}
Let $B$  be an adjacency basis  of  a graph $G$ of order $n\ge 7$.  If  $\dim_A(G)=3$, then for any family ${\cal H}\subseteq {\cal  G}_B(G)$, $$\Sd_A({\cal H})=3.$$
\end{theorem}

The family ${\cal H}$ shown in Figure~\ref{FigSubFamilyC8} is an example of the previous result.

\section{Families of join graphs}
\label{sectFamJoins}

For a graph family  ${\cal H}=\{H_1,H_2,\ldots,H_{k}\}$, defined on common vertex set  $V$,   and the graph $K_1=\langle v \rangle$, $v \notin V$, we define the family $$K_1+{\cal H}=\{K_1+H_i:\;1 \leq i \leq k\}.$$
Notice that since for any   $H_i \in {\cal H}$ the graph $K_1+H_i$ has diameter two, $$\Sd(K_1+{\cal H})=\Sd_A(K_1+{\cal H}).$$

\begin{theorem}\label{relSimAdjDimK1JoinFamily}
Let ${\cal G}$ be a family of non-trivial graphs on a common vertex set $V$. If for every simultaneous adjacency basis $B$ of ${\cal G}$ there exist $G \in {\cal G}$ and $x\in V$ such that $B\subseteq N_{G}(x)$, then $$\Sd(K_1+{\cal G})=\Sd_A({\cal G})+1.$$
Otherwise, $$\Sd(K_1+{\cal G})=\Sd_A({\cal G}).$$
\end{theorem}

\begin{proof}
Let $V(K_1)=\{v_0\}$.
Suppose that for every simultaneous adjacency basis $B$ of ${\cal G}$ there exist $G \in {\cal G}$ and $x\in V$ such that $B\subseteq N_{G}(x)$. In this case, first notice that for every pair of different vertices $u,v \in V$ we have that  $d_{K_1+G,2}(u,v_0)=d_{K_1+G,2}(v,v_0)=1$, so $v_0$ does not distinguish any pair of vertices. In consequence, a simultaneous metric basis of $K_1+{\cal G}$ must contain at least as many vertices as a simultaneous adjacency basis of ${\cal G}$. Secondly, since $B \subseteq N_{K_1+G}(v_0)$ and $   B\subseteq N_{K_1+G}(x)$, a simultaneous metric basis of $K_1+{\cal G}$ must additionally contain some vertex $v \in (V-N_{G}(x))\cup\{v_0\}$, so $\Sd(K_1+{\cal G}) \geq \Sd_A({\cal G})+1$. Let $B$ be a simultaneous adjacency basis of ${\cal G}$ and let $B'=B\cup\{v_0\}$ and $G'\in {\cal G}$. For every pair of different vertices $u,v \in V(K_1+G')-B'$, there exists a vertex $z \in B\subset B'$ such that $d_{K_1+G',2}(u,z)=d_{G',2}(u,z) \neq d_{G',2}(v,z)=d_{K_1+G',2}(v,z)$, so $B'$ is a simultaneous metric generator for $K_1+{\cal G}$ and, as a result,  $\Sd(K_1+{\cal G}) \leq |B'| = |B|+1= \Sd_A({\cal G})+1$. Consequently, $\Sd(K_1+{\cal G})=\Sd_A({\cal G})+1$.

Now suppose that there exists a simultaneous adjacency basis $B$ of ${\cal G}$ such that $B \nsubseteq N_{G}(x)$ for  every $G \in {\cal G}$ and every $x\in V$. In this case, first recall that a simultaneous metric basis of $K_1+{\cal G}$ must contain as many vertices as a simultaneous adjacency basis of ${\cal G}$, so $\Sd(K_1+{\cal G}) \geq \Sd_A({\cal G})$. As above, for every pair of different vertices $u,v\in V-B$, there exists a vertex $z \in B$ such that $d_{K_1+G,2}(u,z)=d_{G,2}(u,z) \neq d_{G,2}(v,z)=d_{K_1+G,2}(v,z)$. Now, for any $u\in V-B$ there exists $u'\in B-N_{G}(u)$ such that $d_{K_1+G,2}(u,u')=2 \neq 1 =d_{K_1+G,2}(v_0,u')$. 
Hence, $B$ is also a simultaneous metric generator for $K_1+{\cal G}$ and, consequently $\Sd(K_1+{\cal G})\le |B|=\Sd_A({\cal G})$. Therefore, $\Sd(K_1+{\cal G})=\Sd_A({\cal G})$.
\end{proof}

Since $K_t+G=K_1+(K_{t-1}+G)$ for any  $t \ge 2$, the previous result can be generalized as follows.

\begin{corollary}\label{relSimAdjDimKtJoinFamily}
Let ${\cal G}$ be a family of non-trivial graphs on a common vertex set $V$ and let $K_t$ be a complete graph of order $t\ge 1$. If for every simultaneous adjacency basis $B$ of ${\cal G}$ there exist $G \in {\cal G}$ and $x\in V$ such that $B\subseteq N_{G}(x)$, then $$\Sd(K_t+{\cal G})=\Sd_A({\cal G})+t.$$
Otherwise, $$\Sd(K_t+{\cal G})=\Sd_A({\cal G})+t-1.$$
\end{corollary}


By Remark \ref{Remark_induced_subgraphs_inG'}  and 
 Theorems \ref{dimAPermFamily} and \ref{relSimAdjDimK1JoinFamily} we deduce the following result.

\begin{theorem}\label{ThmFamiliaK1+H}
Let $B$ be an adjacency basis of a graph $G$ and let ${\cal H}\subseteq {\cal  G}_B(G)$ such that $G\in {\cal  H}$. The following assertions hold:

\begin{enumerate}[{\rm (i)}]
\item If for any adjacency basis $B'$ of $G$, there exists $v\in V(G)$ such that $B' \subseteq  N_G(v)$, then $$\Sd(K_1+{\cal  H})= \dim_A(G)+1.$$

\item If $B\not \subseteq  N_G(v)$ for all $v\in V(G)$, then $$\Sd(K_1+{\cal  H})= \dim_A(G).$$
\end{enumerate} 
\end{theorem}

\begin{proof}
First of all, by Theorem \ref{dimAPermFamily}, $\Sd_A({\cal  H})=\dim_A(G)$ and, as a consequence, every 
 simultaneous adjacency basis of ${\cal H}$, which is also a simultaneous metric basis, is an adjacency basis of $G$. 
Now, if for any adjacency basis $B'$ of $G$, there exists $v\in V(G)$ such that $B' \subseteq  N_G(v)$, then by 
 Theorem    \ref{relSimAdjDimK1JoinFamily}, 
$\Sd(K_1+{\cal  H})=\Sd_A({\cal  H})+1=\dim_A(G)+1.$ Therefore, (i) follows.
On the other hand,  if  $B\not \subseteq  N_G(v)$ for all $v\in V(G)$, then by 
Remark \ref{Remark_induced_subgraphs_inG'} we have that for any $G'\in {\cal  G}_B(G)$, $B\not \subseteq  N_{G'}(v)$ for all $v\in V(G)$. Hence, by Theorem    \ref{relSimAdjDimK1JoinFamily}, 
$\Sd(K_1+{\cal  H})=\Sd_A({\cal  H})=\dim_A(G)$. Therefore, the proof of  (ii) is complete.
\end{proof}

To show some particular cases of the results above, we will state the following two results. 

\begin{remark}{\rm \cite{JanOmo2012}}\label{RemarkPathCycle} For any integer $n\ge 4$,
$$\dim_A(P_n)=\dim_A(C_n)=\displaystyle\left\lfloor\frac{2n+2}{5}\right\rfloor.$$
\end{remark}

\begin{lemma}\label{LemmaDiameter>6orPathorCycle}
Let $G$ be a connected graph. If $D(G)\ge 6$ or $G\in \{P_n,C_n\}$ for $n\ge 7$, or $G$ is a graph of girth $\mathtt{g}(G)\ge 5$ and minimum degree $\delta(G)\ge 3$, then for every adjacency generator $B$ for $G$ and every $v\in V(G)$, $B\not \subseteq N_G(v).$
\end{lemma}
 
\begin{proof}
Let $B$ be an adjacency generator for $G$.  First, suppose that there exists $v\in V(G)$ such that  $B\subseteq N_G(v)$. Since 
 $B$ is an adjacency generator for  $G$,   either $B$ is a dominating set or there exists exactly one vertex $u\in V(G)-B$ which is not dominated by $B$. In the first case, $D(G)\le 4$ and in the second one, either  $D(G)\le 5$ or $u$ is an isolated vertex. 
Hence,   if $D(G)\ge 6$, then $B\not \subseteq N_G(v)$.

Now, assume that $\delta(G)\ge 3$.    Let $v\in V(G)$, $u\in N_G(v)$ and $x,y\in N_G(u)-\{v\}$. If $\mathtt{g}(G)\ge 5$, then no vertex $z\in N_G[v]$ distinguishes $x$ from $y$ and, since $B$ is an adjacency generator for $G$, there exists $z'\in B-N_G[v]$ which distinguishes them. Thus, $B\not \subseteq N_G(v)$. 

Finally, if $G\in \{P_n,C_n\}$ for $n\ge 7$, then by Remark \ref{RemarkPathCycle}  we have $|B|\ge \dim_A(G)=\displaystyle\left\lfloor\frac{2n+2}{5}\right\rfloor \ge 3$ and, since $G$ has maximum degree two, the result follows.    
\end{proof}

According to Lemma \ref{LemmaDiameter>6orPathorCycle},  Theorem \ref{relSimAdjDimK1JoinFamily} immediately leads to the following result. 

\begin{proposition}
Let ${\cal G}$ be a family of  graphs on a common vertex set $V$ of cardinality $|V|\ge 7$. If every $G\in {\cal G}$ is a path graph, or a cycle graph, or $D(G)\ge 6$, or   $\mathtt{g}(G)\ge 5$ and   $\delta(G)\ge 3$, then   $$\Sd(K_1+{\cal G})=\Sd_A({\cal G}).$$
\end{proposition}

Theorem \ref{ThmFamiliaK1+H} and Lemma \ref{LemmaDiameter>6orPathorCycle}   immediately lead to the following result. 

\begin{proposition}
Let $G$ be a graph of order $n\ge 7$ and let $B$ be an adjacency basis of $G$. If $G$ is a path graph, a cycle graph, $D(G)\ge 6$, or $\mathtt{g}(G) \ge 5$ and $\delta(G) \ge 3$, then for any family ${\cal H} \subseteq {\cal G}_B(G)$ such that $G\in {\cal H} $,  $$\Sd(K_1+{\cal H})=\dim_A(G).$$
\end{proposition}

We now discuss particular cases where $\Sd(K_1+{\cal G})=\Sd_A({\cal G})+1$. First, consider a graph family ${\cal G}=\{G_1,G_2,\ldots,G_k\}$, defined on a common vertex set of cardinality $n$, such that $G_i \cong K_n$ for some $i \in \{1,\ldots,k\}$. Since $K_1+K_n=K_{n+1}$, we have that $\Sd(K_1+{\cal G})=n=\Sd_A({\cal G})+1$. Now recall the families ${\cal K}(V)$ of star graphs defined in Section~\ref{sectBasicResults}. The following result holds.

\begin{proposition}\label{partCaseK1PlusStars}
For every finite set $V$ of size $|V|\ge 4$, $\Sd(K_1 + {\cal K}(V))=\Sd_A({\cal K}(V))+1$.
\end{proposition} 

\begin{proof}
Every simultaneous adjacency basis $B$ of ${\cal K}(V)$ has the form $V-\{v_i\}$, $i \in \{1,\ldots,n\}$. In $K_{1,n-1}^i$, we have that $B \subseteq N_{K_{1,n-1}^i}(v_i)$, so the result is deduced by Theorem~\ref{relSimAdjDimK1JoinFamily}.
\end{proof}

For two graph families ${\cal G}=\{G_1,G_2,\ldots,G_{k_1}\}$ and ${\cal H}=\{H_1,H_2,\ldots,H_{k_2}\}$, defined on common vertex sets $V_1$ and $V_2$, respectively, such that $V_1 \cap V_2 = \emptyset$, we define the family $${\cal G}+{\cal H}=\{G_i+H_j:\;1 \leq i \leq k_1, 1 \leq j \leq k_2\}
.$$
Notice that, since for any $G_i \in {\cal G}$ and   $H_j \in {\cal H}$ the graph $G_i+H_j$ has diameter two, $$\Sd({\cal G}+{\cal H})=\Sd_A({\cal G}+{\cal H}).$$

\begin{theorem}\label{SimAdjDimJoinsFamilyCase1}
Let ${\cal G}$ and ${\cal H}$ be two families of non-trivial graphs on common vertex sets $V_1$ and $V_2$, respectively. If there exists a simultaneous adjacency basis $B$ of ${\cal G}$ such that for every  $G \in {\cal G}$ and every $g\in V_1$, $B\not\subseteq N_G(g)$,   
 then $$\Sd({\cal G}+{\cal H})=\Sd_A({\cal G})+\Sd_A({\cal H}).$$
\end{theorem}

\begin{proof}
Let $B$ be a simultaneous adjacency basis of ${\cal G}$ such that $B\nsubseteq N_G(u)$ for every $u\in V_1$, and let  $B'$ be a simultaneous adjacency basis of ${\cal H}$. We claim that the set $S=B \cup B'$ is a simultaneous metric generator for ${\cal G+H}$. Consider a pair of different vertices $u,v\in (V_1 \cup V_2)-S$. If $u,v \in V_1$, then there exists $x \in B$ that distinguishes them in every $G \in {\cal G}$. An analogous situation occurs for $u,v \in V_2$. If $u \in V_1$ and $v \in V_2$, since $B\not\subseteq N_G(u)$, there exists $x \in B$ such that $d_{G+H,2}(u,x)=2 \neq 1=d_{G+H,2}(v,x)$ for every $G\in {\cal G}$ and  $H\in {\cal H}$. Thus, $S$ is a simultaneous metric generator for ${\cal G}+{\cal H}$ and, as a consequence, $\Sd({\cal G}+{\cal H}) \leq \vert S \vert = \vert B \vert + \vert B' \vert = \Sd_A({\cal G})+\Sd_A({\cal H})$.

To prove that $\Sd({\cal G}+{\cal H}) \geq \Sd_A({\cal G})+\Sd_A({\cal H})$, consider a simultaneous metric basis $W$ of ${\cal G}+{\cal H}$. Let $W_1=W \cap V_1$ and let $W_2=W\cap V_2$. Let $G\in {\cal G}$ and $H\in {\cal H}$. No pair of different vertices $u,v \in V_2-W_2$ is distinguished in $G+H$ by any vertex from $W_1$, whereas no pair of different vertices $u,v \in V_1-W_1$ is distinguished in $G+H$ by any vertex from $W_2$, so $W_1$ is a simultaneous adjacency generator for ${\cal G}$ and $W_2$ is a simultaneous adjacency generator for ${\cal H}$. Thus, $\Sd({\cal G}+{\cal H})=|W|= \vert W_1 \vert + \vert W_2 \vert \geq \Sd_A({\cal G})+\Sd_A({\cal H})$.
\end{proof}
By Lemma \ref{LemmaDiameter>6orPathorCycle} we deduce the following   consequence  of Theorem \ref{SimAdjDimJoinsFamilyCase1}.

\begin{corollary}
Let ${\cal G}$ be a family of  graphs on a common vertex set $V$ of cardinality $|V|\ge 7$. If every $G\in {\cal G}$ is a path graph, a cycle graph, $D(G)\ge 6$, or $\mathtt{g}(G) \ge 5$ and $\delta(G) \ge 3$, then for any family ${\cal H} $ of non-trivial graphs on a common vertex set,  $$\Sd({\cal G + H})=\Sd_A({\cal G})+\Sd_A({\cal H}).$$
\end{corollary}

Theorems \ref{dimAPermFamily} and  \ref{SimAdjDimJoinsFamilyCase1} and Lemma \ref{LemmaDiameter>6orPathorCycle} lead to the next result.
 
\begin{theorem}
Let $G$ be a graph of order $n\ge 7$ and let $B$ be an adjacency basis of $G$. If $G$ is a path graph, or a cycle graph, or $D(G)\ge 6$, or   $\mathtt{g}(G)\ge 5$ and   $\delta(G)\ge 3$, then for any family ${\cal G}_1 \subseteq {\cal G}_B(G)$ such that $G\in {\cal G}_1 $  and any family ${\cal H} $ of non-trivial graphs on a common vertex set, $$\Sd({\cal G}_1+{\cal H})=\dim_A(G)+\Sd_A({\cal H}).$$
\end{theorem}

The ideas introduced in Theorem~\ref{dimAPermFamily} allow us to define large families composed by subgraphs of a join graph $G+H$, which may be seen as the result of a relaxation of the join operation, in the sense that not every pair of nodes $u \in V(G)$, $v \in V(H)$, must be linked by an edge, yet any adjacency basis of $G+H$ is a simultaneous adjacency generator for the family, and thus a simultaneous metric generator. Since for any adjacency basis $B$ of $G+H$, the family ${\cal R}_B$ defined in the next result is a subfamily of ${\cal G}_B(G+H)$, the result follows directly from Theorem \ref{dimAPermFamily}.

\begin{corollary}\label{famJoinsEdgesRemoved}
Let $G$ and $H$ be two non-trivial graphs and let $B$ be an adjacency basis of $G+H$. Let $E'=\{uv \in E(G+H) :\; u \in V(G)-B,\;  v \in V(H)-B\}$ and let ${\cal R}_B=\{R_1,R_2,\ldots,R_k\}$ be a graph family, defined on the common vertex set $V(G+H)$, such that, for every $i \in \{1,\ldots,k\}$, $E(R_i)=E(G+H)-E_i$, for some edge subset  $E_i \subseteq E'$. Then $$\Sd({\cal R}_B) \leq \dim(G+H).$$
\end{corollary}

As the next result shows, it is possible to obtain families composed by join graphs of the form $G'+H'$, where $G'$ and $H'$ are the result of applying modifications to $G$ and $H$, respectively, in such a way that any adjacency basis of $G+H$ is a simultaneous adjacency generator for the family, and thus a simultaneous metric generator.

\begin{corollary}\label{famJoinsFactorsModified}
Let $G$ and $H$ be two non-trivial graphs and let $B$ be an adjacency basis of $G+H$. Let $B_1=B\cap V(G)$ and $B_2=B\cap V(H)$. Then for any family ${\cal H}\subseteq {\cal G}_{B_1}(G)+{\cal G}_{B_2}(H)$,
$$\Sd({\cal H}) \leq \dim(G+H).$$
Moreover, if $G+H\in {\cal H}$, then 
$$\Sd({\cal H}) = \dim(G+H).$$
\end{corollary}

\begin{proof}
The result is a direct consequence of Theorem \ref{dimAPermFamily}, as ${\cal G}_{B_1}(G)+{\cal G}_{B_2}(H)\subseteq {\cal G}_B(G+H)$.
\end{proof}

Given two families  ${\cal G}$ and ${\cal H}$ of non-trivial graphs on common vertex sets $V_1$ and $V_2$, respectively, we define  ${\cal B}({\cal G})$ and ${\cal B}({\cal H})$ as the sets composed by all simultaneous adjacency bases of ${\cal G}$ and ${\cal H}$, respectively. For a simultaneous adjacency basis $B \in {\cal B}({\cal G})$, consider the set 
$$P(B)=\{u \in V_1 :\; B\subseteq N_{G}(u)\textrm{ for some } G\in {\cal G}\}.$$

Similarly, for a simultaneous adjacency basis $B' \in {\cal B}({\cal H})$, consider the set $$Q(B')=\{v \in V_2 :\; B'\subseteq N_{H}(v)\textrm{ for some } H\in {\cal H}\}.$$

Based on the definitions of $P(B)$ and $Q(B')$, we define the parameter $\psi({  \cal G},{\cal H})$ as

\begin{displaymath}
\psi({  \cal G},{\cal H})=\underset{B' \in {\cal B}({\cal H})}{\underset{B \in {\cal B}({\cal G}),}{\min}}\left\{\vert P(B) \vert, \vert Q(B') \vert\right\}.
\end{displaymath}

The following result holds.

\begin{theorem}\label{SimAdjDimJoinsFamilyCase2}
Let ${\cal G}$ and ${\cal H}$ be two families of non-trivial graphs on common vertex sets $V_1$ and $V_2$, respectively. If for every simultaneous adjacency basis $B_1$ of ${\cal G}$ there exists $G \in {\cal G}$ and $g\in V_1$ such that  $B_1\subseteq N_G(g)$ and for every simultaneous adjacency basis $B_2$ of ${\cal H}$ there exists  $H \in {\cal H}$ and $h\in V_2$ such that $B_2\subseteq N_H(h)$,  then $$\Sd_A({\cal G})+\Sd_A({\cal H})+1 \leq \Sd({\cal G}+{\cal H}) \leq \Sd_A({\cal G})+\Sd_A({\cal H})+\psi({\cal G},{\cal H}).$$
\end{theorem}

\begin{proof}
We first address the proof of the lower bound. Let $W$ be a simultaneous metric  basis of ${\cal G}+{\cal H}$ and let $W_1=W \cap V_1$ and $W_2=W \cap V_2$.  Let $G\in {\cal G}$ and $H\in {\cal H}$. Since no pair of different vertices $u,v \in V_2-W_2$ is distinguished by any vertex in $W_1$, whereas no pair of different vertices $u,v \in V_1-W_1$ is distinguished by any vertex in $W_2$, we conclude that $W_1$ is an adjacency generator for $G$ and $W_2$ is an adjacency generator for $H$. Hence, $W_1$ is  a simultaneous adjacency generator for  ${\cal G}$ and  $W_2$ is   a simultaneous  adjacency generator for ${\cal H}$.  If $W_1$ is a simultaneous adjacency basis of ${\cal G}$ and $W_2$ is a simultaneous adjacency basis of ${\cal H}$, then under the assumptions of this theorem, for at least one graph $G+H\in {\cal G+H}$ there exist $x \in V_1-W_1$ and $y \in V_2-W_2$, such that $W\subseteq N_{G+H}(x)$ and $W\subseteq N_{G+H}(y)$, which is a contradiction. Thus, $W_1$ is not a simultaneous adjacency basis of ${\cal G}$ or $W_2$ is not a simultaneous  adjacency basis of ${\cal H}$. Hence,  $\vert W_1 \vert \geq \Sd_A({\cal G})+1$ or $\vert W_2 \vert \geq \Sd_A({\cal H})+1$. In consequence, we have that $\Sd({\cal G}+{\cal H})=|W|= \vert W_1 \vert + \vert W_2 \vert \geq \Sd_A({\cal G})+\Sd_A({\cal H})+1$.

We now address the proof of the upper bound. Let $B_1$ and $B_2$ be simultaneous adjacency bases of ${\cal G}$ and ${\cal H}$, respectively, for which $\psi({  \cal G},{\cal H})$ is obtained. Assume, without loss of generality, that $\vert P(B_1) \vert \leq \vert Q(B_2) \vert$. Let $S=B_1 \cup B_2 \cup P(B_1)$. We claim that $S$ is a simultaneous metric generator for ${\cal G}+{\cal H}$. To show this,  we differentiate two cases for any $G \in {\cal G} $ and $H \in {\cal H}$.

Case 1. There exists $g\in V_1$ such that $B_1\subseteq N_G(g)$.
We claim that the set $S'=B_1 \cup B_2 \cup \{g\}\subseteq S$ is a metric generator for $G+H$. To see this, we only need to check that for any $u\in V_1-(B_1\cup \{g\})$ and $v\in V_2-B_2$ there exists $s\in S'$ which distinguishes them, as $B_1$ and $B_2$ are adjacency generators for $G$ and $H$, respectively. 
That is, since $g$  is the sole vertex in $V_1$ satisfying $N_G(g)\supseteq B_1$, for any $u\in V_1-(B_1\cup \{g\})$ and $v\in V_2-B_2$ there exists $s\in B_1\subset S'$ such that $d_{G+H,2}(u,s)=2 \ne  1= d_{G+H,2}(v,s)$. Hence, the set $S'\subseteq S$ is a metric generator for $G+H$. 

Case 2. No vertex 
$g\in V_1$ satisfies  $B_1\subseteq N_G(g)$. In this case,  the set $S'=B_1 \cup B_2 \subseteq S$ is a metric generator for $G+H$, as $B_1$ and $B_2$ are adjacency generators for $G$ and $H$, respectively, and
for any $u\in V_1-B_1 $ and $v\in V_2-B_2$ there exists $s\in B_1\subset S'$ such that $d_{G+H,2}(u,s)=2 \ne  1= d_{G+H,2}(v,s)$. 

Therefore, $S$ is a simultaneous metric generator for ${\cal G}+{\cal H}$, so $\Sd({\cal G}+{\cal H}) \leq \vert S \vert = \vert B_1 \vert + \vert B_2 \vert + \vert P(B_1) \vert = \Sd_A({\cal G})+\Sd_A({\cal H})+\psi({  \cal G},{\cal H})$.
\end{proof}

As the following corollary shows, the inequalities above are tight. 

\begin{corollary}\label{partCaseFansOrWheelsPlusFansOrWheels}
Let ${\cal G}=\{G_1, G_2, \ldots, G_k\}$ and ${\cal G}'=\{G'_1, G'_2, \ldots, G'_{k'}\}$ be families composed by paths and/or cycle graphs on common vertex sets $V$ and $V'$ of sizes $n \ge 7$ and $n' \ge 7$, respectively. Let $u,v \notin V\cup V'$, $u\ne v$,  and let ${\cal H}=\{\langle u \rangle + G_1, \langle u \rangle + G_2, \ldots, \langle u \rangle + G_k\}$ and ${\cal H}'=\{\langle v \rangle + G'_1, \langle v \rangle + G'_2, \ldots, \langle v \rangle + G'_{k'}\}$.
Then, $$\Sd({\cal H}+{\cal H}')=\Sd_A({\cal H})+\Sd_A({\cal H}')+1.$$
\end{corollary}

\begin{proof} By Lemma \ref{LemmaDiameter>6orPathorCycle} we have that
for every simultaneous adjacency generator $B$ for $G\in {\cal G}$ and every $v\in V(G)$, $B\not \subseteq N_G(v).$ Hence, as we have shown in the proof of Theorem  \ref{relSimAdjDimK1JoinFamily},   any simultaneous adjacency basis of ${\cal G}$ is a simultaneous adjacency basis of $K_1+{\cal G}\cong \langle u\rangle+{\cal G}={\cal H}$ and vice versa. So, for any simultaneous adjacency basis $B$ of ${\cal H}$ we have that $P(B)=\{u\}$. Analogously, for any simultaneous adjacency basis $B'$ of ${\cal H}'$, we have  $Q(B')=\{v\}$ and so $\psi({  \cal H},{\cal H'})=1$.
\end{proof}

Notice that the result above can be extended to any pair of graph families ${\cal G}$ and ${\cal G}'$ satisfying the premises of   Lemma \ref{LemmaDiameter>6orPathorCycle}.

\section{Families of standard lexicographic product graphs}
\label{sectFamStdLex}

Note that the lexicographic product of two graphs is not a commutative operation. Moreover, $G\circ H$ is a connected graph if and only if $G$ is connected. We would point out the following known result.

\begin{claim}{\rm \cite{Hammack2011}}\label{basictoolLexicographic} Let $G$ and $H$ be two non-trivial graphs such that $G$ is connected. Then the following assertions hold for any  $a,c\in V(G)$ and $b,d\in V(H)$ such  that $a\ne c$.
\begin{enumerate}[{\rm (i)}]
\item $N_{G\circ H}(a,b)=\left(\{a\}\times  N_H(b)\right)\cup \left( N_G(a)\times  V(H)\right)$.
\item  $d_{G\circ H}((a,b),(c,d)) = d_{G}(a,c)$
\item   $d_{G\circ H}((a,b),(a,d)) = d_{H,2}(b,d)$.
\end{enumerate}
\end{claim}

Several results on the metric dimension of the lexicographic product $G \circ H$ of two graphs $G$ and $H$, and its relation to the adjacency dimension of $H$, are presented in \cite{JanOmo2012}. In this section, we study the simultaneous metric  dimension of several families composed by lexicographic product graphs, exploiting the simultaneous adjacency dimension as an important tool.

First, we introduce some necessary notation. Let $S$ be a subset of $V(G \circ H)$. The \emph{projection} of $S$ onto $V(G)$ is the set $\{u:\;(u ,v) \in S\}$, whereas the projection of $S$ onto $V(H)$ is the set $\{v :\;(u,v) \in S\}$. We define the \textit{twin  equivalence relation} ${\cal T}$ on $V(G)$ as follows:
$$x {\cal T} y \Longleftrightarrow N_G[x]=N_G[y] \; \; \mbox{\rm or } \; N_G(x)=N_G(y).$$

In what follows, we will denote the equivalence class of vertex $x$ by $x^*=\{y\in V(G):\; y{\cal T} x\}$.
Notice that every equivalence class may be a singleton set, a clique of size at least two of $G$ or an independent set of size at least two of $G$. We will refer to equivalence classes which are non-singleton cliques as \emph{true twin equivalence classes} and to equivalence classes which are non-singleton independent sets as \emph{false twin equivalence classes}. From now on, $T(G)$ denotes the set of all true twin equivalence classes in $V(G)$, whereas $F(G)$ denotes the set of all false twin equivalence classes in $V(G)$. Finally, $V_T(G)$ and $V_F(G)$ denote the sets of vertices belonging to true and false twin equivalence classes, respectively.

For two graph families ${\cal G}=\{G_1,G_2,\ldots,G_{k_1}\}$ and ${\cal H}=\{H_1,H_2,\ldots,H_{k_2}\}$, defined on common vertex sets $V_1$ and $V_2$, respectively, we define the family $${\cal G}\circ{\cal H}=\{G_i \circ H_j:\;1 \leq i \leq k_1, 1 \leq j \leq k_2\}.$$

In particular, if ${\cal G}=\{G\}$ we will use the notation $G \circ {\cal H}$.

Our first result allows to extend any result on the simultaneous adjacency dimension of ${\cal G}\circ{\cal H}$ to the simultaneous metric dimension, and \textit{vice versa}.

\begin{theorem}\label{mainLemmaLexProd}
Let $G$ be a connected graph and let $H$ be a non-trivial graph. Then, every metric generator for $G \circ H$ is also an adjacency generator, and \textit{vice versa}.
\end{theorem}

\begin{proof}
By definition, every adjacency generator for $G \circ H$ is also a metric generator, so we only need to prove that any metric generator for $G \circ H$ is also an adjacency generator. Let $S$ be a metric generator for $G \circ H$. For a vertex $u_i \in V(G)$, let $R_i=\{u_i\} \times V(H)$. Notice that $R_i \cap S\ne \emptyset$, for every $u_i \in V(G)$, as no vertex outside of $\{u_i\}\times V(H)$ distinguishes pairs of vertices in $\{u_i\}\times V(H)$.
 We differentiate the following cases for two different vertices $(u_i,v_r),(u_j,v_s)\in V(G \circ H)-S$:
\begin{enumerate}
\item $i = j$. In this case, no vertex from $R_x \cap S$, $x \neq i$, distinguishes $(u_i,v_r)$ and $(u_j,v_s)$, so there exists $(u_i,v) \in R_i \cap S$ such that $d_{G \circ H,2}((u_i,v_r),(u_i,v))=d_{G \circ H}((u_i,v_r),(u_i,v)) \neq d_{G \circ H}((u_j,v_s),(u_i,v))=d_{G \circ H,2}((u_j,v_s),(u_i,v))$.
\item $u_i$ and $u_j$  are true twins $(i \neq j)$. Here, no vertex from $R_x \cap S$, $x \notin \{i,j\}$, distinguishes $(u_i,v_r)$ and $(u_j,v_s)$, so there exists $(u_i,v)\in R_i \cap S$ such that $d_{G \circ H,2}((u_i,v_r),(u_i,v))=d_{G \circ H}((u_i,v_r),(u_i,v))=2 \neq 1=d_{G \circ H}((u_j,v_s),(u_i,v))=d_{G \circ H,2}((u_j,v_s),(u_i,v))$, or there exists $(u_j,v)\in R_j \cap S$ such that $d_{G \circ H,2}((u_i,v_r),(u_j,v))=d_{G \circ H}((u_i,v_r),(u_j,v))=1 \neq 2=d_{G \circ H}((u_j,v_s),(u_j,v))=d_{G \circ H,2}((u_j,v_s),(u_j,v))$.
\item   $u_i$ and $u_j$ are false twins ($i \neq j$). As in the previous case, no vertex from $R_x \cap S$, $x \notin \{i,j\}$, distinguishes $(u_i,v_r)$ and $(u_j,v_s)$, so there exists $(u_i,v)\in R_i \cap S$ such that $d_{G \circ H,2}((u_i,v_r),(u_i,v))=d_{G \circ H}((u_i,v_r),(u_i,v))=1 \neq 2=d_{G \circ H}((u_j,v_s),(u_i,v))=d_{G \circ H,2}((u_j,v_s),(u_i,v))$, or there exists $(u_j,v)\in R_j \cap S$ such that $d_{G \circ H,2}((u_i,v_r),(u_j,v))=d_{G \circ H}((u_i,v_r),(u_j,v))=2 \neq 1=d_{G \circ H}((u_j,v_s),(u_j,v))=d_{G \circ H,2}((u_j,v_s),(u_j,v))$.
\item $u_i$ and $u_j$ are not twins. In this case, there exists $u_x \in V(G)-\{u_i,u_j\}$ such that $d_{G,2}(u_i,u_x)\neq d_{G,2}(u_j,u_x)$. Hence, for any $(u_x,v) \in R_x \cap S$ we have that $d_{G \circ H,2}((u_i,v_r),(u_x,v))=d_{G,2}(u_i,u_x)$ $\neq d_{G,2}(u_j,u_x)=d_{G \circ H,2}((u_j,v_s),(u_x,v))$.
\end{enumerate}
In conclusion, $S$ is an adjacency generator for $G \circ H$. The proof is complete.
\end{proof}

\begin{corollary}\label{mainLemmaFamLexProd}
For any  connected graph   and any  non-trivial graph $H$, $$\dim(G\odot H)=\dim_A(G\odot H).$$
In general, for every family ${\cal G}$ composed by connected graphs on a common vertex set, and every family ${\cal H}$ composed by non-trivial graphs on a common vertex set, $$\Sd({\cal G}\circ{\cal H})=\Sd_A({\cal G}\circ{\cal H}).$$
\end{corollary}

 We would point out that the equalities above hold, even for lexicographic product graphs of diameter greater than two.

The following result, presented in \cite{JanOmo2012}, gives a lower bound on $\dim(G \circ H)$, which depends on the order of $G$ and $\dim_A(H)$.

\begin{theorem}
\label{lowerBoundDimLexAdjH}
{\rm \cite{JanOmo2012}} Let $G$ be a connected graph of order $n$ and let $H$ be a non-trivial graph. Then $\dim(G \circ H) \geq n\cdot\dim_A(H)$.
\end{theorem}

We now generalise the previous result for families composed by lexicographic product graphs.

\begin{theorem}\label{lowerBoundSdLexSdAH}
Let ${\cal G}$ be a family of connected graphs on a common vertex set $V_1$ and let ${\cal H}$ be a family of non-trivial graphs on a common vertex set $V_2$. Then $$\Sd({\cal G}\circ{\cal H}) \geq \vert V_1 \vert \cdot \Sd_A({\cal H}).$$
\end{theorem}

\begin{proof}
It was shown  in \cite{JanOmo2012} that if $S'$ is a metric generator for $G \circ H$, and $R_i = \{u_i\} \times V(H)$ for some $u_i \in V(G)$, then $S' \cap R_i$ resolves all vertex pairs in $R_i$, and the projection of $S' \cap R_i$ onto $V(H)$ is an adjacency generator for $H$. Following an analogous reasoning, consider a simultaneous metric generator $S$ for ${\cal G}\circ{\cal H}$, and let $R_i=\{u_i\} \times V_2$ for some $u_i \in V_1$. We have that the projection of $S \cap R_i$ onto $V_2$ is an adjacency generator for every $H \in {\cal H}$ and, in consequence, a simultaneous adjacency generator for ${\cal H}$, so $|R_i \cap S|\ge \Sd_A({\cal H})$. Thus, $\Sd({\cal G}\circ{\cal H})=|S|=\displaystyle{\sum_{u_i \in V_1}|R_i \cap S|} \geq |V_1|\cdot\Sd_A({\cal H})$.
\end{proof}

In order to present our next results, we introduce some additional definitions. For a graph family ${\cal G}$, defined on a common vertex set $V$, let $V_M({\cal G})=\{u :\; u \in V_T(G), u \in V_F(G') \textrm{ for some } G,G' \in {\cal G}\}$. Moreover, for a family ${\cal H}$ composed by non-trivial graphs on a common vertex set $V'$, let ${\cal B}_1({\cal H})$ be the set of simultaneous adjacency bases $B$ of ${\cal H}$ satisfying $B \nsubseteq N_{H}(v)$ for every $H \in {\cal H}$ and every $v \in V'$, and let ${\cal B}_2({\cal H})$ be the set of simultaneous adjacency bases of ${\cal H}$ that are also dominating sets of every $H \in {\cal H}$. Finally, we define the parameter $$\zeta({\cal H})=\min \left\{k_2,\underset{\underset{B_2 \in {\cal B}_2({\cal H})}{_{B_1 \in {\cal B}_1({\cal H})}}}{\min} \left\{\vert B_2 - B_1 \vert\right\}\right\}.$$

With these definitions in mind, we give the next result.

\begin{theorem}\label{upperBoundFamilyLex1FamGFamH}
Let ${\cal G}=\{G_1,G_2,\ldots,G_{k_1}\}$ be a family of connected graphs on a common vertex set $V_1$, let ${\cal H}=\{H_1,H_2,\ldots,H_{k_2}\}$ be a family of non-trivial graphs, defined on a common vertex set $V_2$, such that ${\cal B}_1({\cal H})$ and ${\cal B}_2({\cal H})$ are not empty, and let $\overline{\cal H}=\{\overline{H}_1,\overline{H}_2\,\ldots,\overline{H}_{k_2}\}$. If $V_M({\cal G})=\emptyset$ or ${\cal B}_1({\cal H}) \cap {\cal B}_2({\cal H}) \neq \emptyset$, then
\begin{equation}\label{eqFirstEqSdA}
\Sd({\cal G} \circ {\cal H})=\Sd({\cal G} \circ \overline{\cal H})=\vert V_1 \vert \cdot \Sd_A({\cal H}).
\end{equation}

\noindent
Otherwise,
\begin{equation}\label{eqIneqSdA}
\vert V_1 \vert \cdot \Sd_A({\cal H})+\vert V_M({\cal G}) \vert\leq\Sd({\cal G} \circ {\cal H})=\Sd({\cal G} \circ \overline{\cal H}) \leq \vert V_1 \vert \cdot \Sd_A({\cal H}) + \zeta({\cal H}) \cdot \vert V_M({\cal G}) \vert.
\end{equation}
\end{theorem}

\begin{proof}
We first assume that $V_M({\cal G})=\emptyset$. By Theorem~\ref{lowerBoundSdLexSdAH}, we have that $\Sd({\cal G} \circ {\cal H}) \geq |V_1|\cdot\Sd_A({\cal H})$. Thus, it only remains to prove that $\Sd({\cal G} \circ {\cal H}) \le |V_1|\cdot\Sd_A({\cal H})$. To this end, consider the partition $\{V'_1,V''_1\}$ of $V_1$, where $V'_1=\{u:\;u \in V_T(G)\textrm{ for some }G \in {\cal G}\}$, and a pair of simultaneous adjacency bases $B_1 \in {\cal B}_1({\cal H})$ and $B_2 \in {\cal B}_2({\cal H})$. 
Consider the set $$S=\left(V'_1 \times B_1\right) \cup \left(V''_1 \times B_2\right).$$
It  was shown in \cite{JanOmo2012} that a set constructed in this manner, considering ${\cal G}=\{G\}$ and  ${\cal H}=\{H\}$, is a metric generator for $G \circ H$. Following an analogous reasoning, we shall deduce that $S$ is also a metric generator for every $G \circ H\in    {\cal G}\circ{\cal H}$, and thus a simultaneous metric generator for ${\cal G} \circ {\cal H}$. For the sake of thoroughness of our discussion, we elaborate the four cases for two different vertices $(u_i,v_r),(u_j,v_s)\in V(G\circ H)-S$:
\begin{enumerate}
\item $i=j$. In this case, $r\ne s$. Let $R_i=\{u_i\} \times V_2$. Since $S\cap R_i=\{u_i\}\times B_1$ or $S\cap R_i=\{u_i\}\times B_2$ and both $B_1$ and $B_2$ are adjacency generators for $H$, there exists $v \in B_1$ such that $d_{H,2}(v,v_r)\ne d_{H,2}(v,v_s)$, or there exists $v \in B_2$ such that $d_{H,2}(v,v_r)\ne d_{H,2}(v,v_s)$. Since for every $(u_i,v_r),(u_i,v_s)\in R_i$ we have that $d_{G \circ H,2}((u_i,v_r),(u_i,v_s))=d_{H,2}(v_r,v_s)$, we conclude that at least one element from $S$ distinguishes $(u_i,v_r)$ and $(u_i,v_s)$.
\item $i\ne j$ and $u_i,u_j$ are true twins. Here, since $B_1 \nsubseteq N_{H}(v_r)$, there exists $v\in B_1$ such that $d_{H,2}(v_r,v)=2$. Thus, $d_{G \circ H,2}((u_i,v_r),(u_i,v))=d_{H,2}(v_r,v)=2\ne 1=d_{G,2}(u_j,u_i)=d_{G \circ H,2}((u_j,v_s),(u_i,v))$.
\item $i\ne j$ and $u_i,u_j$ are false twins. Here, since $B_2$ is a dominating set of $H$, there exists $v\in B_2$ such that $d_{H,2}(v_r,v)=1$. Thus, $d_{G \circ H,2}((u_i,v_r),(u_i,v))=d_{H,2}(v_r,v)=1\ne 2=d_{G,2}(u_j,u_i)=d_{G \circ H,2}((u_j,v_s),(u_i,v))$.
\item $i\ne j$ and $u_i,u_j$ are not twins. Here, there exists $u_z \in V_1$ such that $d_{G,2}(u_i,u_z)\ne d_{G,2}(u_j,u_z)$. Since $S \cap R_z\neq\emptyset$, we have that $d_{G \circ H,2}((u_i,v_r),(u_z,v))=d_{G,2}(u_i,u_z)\ne d_{G,2}(u_j,u_z)=d_{G \circ H,2}((u_j,v_s),(u_z,v))$ for every $(u_z,v)\in S$.
\end{enumerate}
Therefore, $S$ is a metric generator for every $G \circ H \in {\cal G} \circ {\cal H}$ and, in consequence, a simultaneous metric generator for ${\cal G} \circ {\cal H}$. Hence, $\Sd({\cal G} \circ {\cal H}) \le \vert S \vert = |V_1|\cdot\Sd_A({\cal H})$ and the equality holds.

We now address the proof of $\Sd({\cal G} \circ \overline{\cal H})=|V_1| \cdot \Sd_A({\cal H})$. As pointed out in \cite{JanOmo2012}, $B_1$ is a dominating set of every $\overline{H} \in \overline{\cal H}$ and $B_2$ satisfies $B_2 \nsubseteq N_{\overline{H}}(v)$ for every $\overline{H} \in \overline{\cal H}$ and every $v \in V_2$. Since $\Sd_A({\cal H})=\Sd_A(\overline{{\cal H}})$, by exchanging the roles of $B_1$ and $B_2$ and proceeding in a manner analogous to the one used for proving that $\Sd({\cal G} \circ {\cal H})\leq |V_1|\cdot\Sd_A({\cal H})$, we obtain that $\Sd({\cal G} \circ \overline{\cal H}) \leq |V_1|\cdot\Sd_A(\overline{\cal H})=|V_1|\cdot\Sd_A({\cal H})$. Since $\Sd({\cal G} \circ \overline{\cal H}) \geq |V_1|\cdot\Sd_A(\overline{\cal H})=|V_1|\cdot\Sd_A({\cal H})$ by Theorem~\ref{lowerBoundSdLexSdAH}, the equality holds.

From now on, we assume that $V_M({\cal G}) \neq \emptyset$ and ${\cal B}_1({\cal H}) \cap {\cal B}_2({\cal H}) \neq \emptyset$. Consider a simultaneous adjacency basis $B \in {\cal B}_1({\cal H}) \cap {\cal B}_2({\cal H})$. By a reasoning analogous to the one previously shown, we have that the set $S=V_1 \times B$ is a metric generator for every $G \circ H \in {\cal G}\circ{\cal H}$ and every $G \circ \overline{H} \in {\cal G}\circ \overline{\cal H}$. Consequently, $S$ is a simultaneous metric generator for ${\cal G} \circ {\cal H}$ and ${\cal G} \circ \overline{\cal H}$, so $\Sd({\cal G} \circ {\cal H}) \le \vert S \vert = \vert V_1 \vert \cdot \Sd_A({\cal H})$ and $\Sd({\cal G} \circ \overline{\cal H}) \le \vert S \vert = \vert V_1 \vert \cdot \Sd_A({\cal H})$. By Theorem~\ref{lowerBoundSdLexSdAH}, $\Sd({\cal G} \circ {\cal H}) \ge \vert V_1 \vert \cdot \Sd_A({\cal H})$ and $\Sd({\cal G} \circ \overline{\cal H}) \ge \vert V_1 \vert \cdot \Sd_A({\cal H})$, so the equalities hold.

From now on, we assume that $V_M({\cal G}) \neq \emptyset$ and ${\cal B}_1({\cal H}) \cap {\cal B}_2({\cal H}) = \emptyset$. Let $B$ be a simultaneous metric basis of ${\cal G} \circ {\cal H}$ and let $B_p=B \cap (\{u_p\} \times V_2)$ for some $u_p \in V_1$. Recall that, as shown in the proof of Theorem~\ref{lowerBoundSdLexSdAH}, the projection of $B_p$ onto $V_2$ is a simultaneous adjacency generator for ${\cal H}$. Let $B'_p$ be the projection onto $V_2$ of some $B_p$ such that $u_p \in V_M({\cal G})$. Suppose, for the purpose of contradiction, that $|B'_p|=\Sd_A({\cal H})$. Let $G \in {\cal G}$ be a graph where $u_p \in V_T(G)$ and let $G' \in {\cal G}$ be a graph where $u_p \in V_F(G')$. We have that there exists $v \in V_2-B'_p$ such that either $B'_p \subseteq N_{H'}(v)$ for some $H' \in {\cal H}$ or $B'_p \cap N_{H''}(v)=\emptyset$ for some $H'' \in {\cal H}$. In the first case, no vertex $(x,y)\in B$ distinguishes in $G \circ H'$ the vertex $(u_p,v)$ from any vertex $(u_t,w)$ such that $u_p$ and $u_t$ are true twins in $G$, whereas in the second case, no vertex $(x,y) \in B$ distinguishes in $G' \circ H''$ the vertex $(u_p,v)$ from any vertex $(u_f,w)$ such that $u_p$ and $u_f$ are false twins in $G'$. In either case, we have a contradiction with the fact that $B$ is a simultaneous metric basis of ${\cal G}\circ{\cal H}$. Thus, for every $u_p \in V_M({\cal G})$, we have that $\vert B_p \vert = \vert B'_p \vert \geq \Sd_A({\cal H})+1$. In conclusion,

\begin{displaymath}
\begin{array}{rcl}
\Sd({\cal G} \circ {\cal H})&=&\vert B \vert = \displaystyle{\sum_{u_p \in V_1-V_M({\cal G})}}\vert B_p \vert + \displaystyle{\sum_{u_p \in V_M({\cal G})}} \vert B_p \vert \geq\\
&\geq&\displaystyle{\sum_{u_p \in V_1-V_M({\cal G})}}\Sd_A({\cal H}) + \displaystyle{\sum_{u_p \in V_M({\cal G})}}(\Sd_A({\cal H})+1)=\\
&=&\vert V_1-V_M({\cal G}) \vert \cdot \Sd_A({\cal H})+\vert V_M({\cal G}) \vert \cdot (\Sd_A({\cal H})+1)=\\
&=&\vert V_1 \vert \cdot \Sd_A({\cal H})+\ \vert V_M({\cal G}) \vert.
\end{array}
\end{displaymath}

In order to prove the upper bound, consider the partition $\{V_M({\cal G}),V'_1,V''_1\}$ of $V_1$, where $V'_1=\{u:\;u \in V_T(G)\textrm{ for some }G \in {\cal G}\}$. Since ${\cal B}_1({\cal H})$ and ${\cal B}_2({\cal H})$ are disjoint, for any $B_1 \in {\cal B}_1({\cal H})$ and $B_2 \in {\cal B}_2({\cal H})$, there exist up to $k_2$ vertices $v_{p_1},v_{p_2},\ldots,v_{p_r} \in V_2-B_1$ such that $B_1 \cap N_{H}(v_{p_i})=\emptyset$ for some $H \in {\cal H}$ and up to $k_2$ vertices $v_{q_1},v_{q_2},\ldots,v_{q_s} \in V_2-B_2$ such that $B_2 \subseteq N_{H}(v_{q_i})$ for some $H \in {\cal H}$. We define the sets $B'_1=B_1 \cup \{v_{p_1},v_{p_2},\ldots,v_{p_r}\}$ and $B'_2=B_2 \cup \{v_{q_1},v_{q_2},\ldots,v_{q_s}\}$, which are simultaneous adjacency generators for ${\cal H}$ that are also dominating sets of every $H \in {\cal H}$ and satisfy $B'_1 \nsubseteq N_{H}(w)$ and $B'_2 \nsubseteq N_{H}(w)$ for every $w \in V_2$ and every $H \in {\cal H}$.

Consider one $B_1 \in {\cal B}_1({\cal H})$  such that $|B'_1|$ is minimum and any $B_2 \in {\cal B}_2({\cal H})$. We define the set $S_1=(V'_1 \times B_1) \cup (V''_1 \times B_2) \cup (V_M({\cal G}) \times B'_1)$. Likewise, consider one $B_2 \in {\cal B}_2({\cal H})$ such that $|B'_2|$ is minimum and any $B_1 \in {\cal B}_1({\cal H})$. We define the set $S_2=(V'_1 \times B_1) \cup (V''_1 \times B_2) \cup (V_M({\cal G}) \times B'_2)$. Finally, consider a pair of simultaneous adjacency bases $B_1 \in {\cal B}_1({\cal H})$ and $B_2 \in {\cal B}_2({\cal H})$ such that $|B_1 \cup B_2|$ is minimum. As $\vert B_1 \vert = \vert B_2 \vert$, we have that $\vert B_2 - B_1 \vert = \vert B_1 - B_2 \vert$ and is also minimum. We define the set $S_3=(V'_1 \times B_1) \cup (V''_1 \times B_2) \cup (V_M({\cal G}) \times (B_1 \cup B_2))$. Now, recall that for every $G \in {\cal G}$ the sets $S=(V_T(G) \times B_1) \cup ((V_1-V_T(G)) \times B_2)$ and $S'=((V_1-V_F(G)) \times B_1) \cup (V_F(G) \times B_2)$ are metric generators for every $G \circ H \in {\cal G}\circ{\cal H}$. Clearly, $S \subseteq S_1$ or $S' \subseteq S_1$, whereas $S \subseteq S_2$ or $S' \subseteq S_2$, and $S \subseteq S_3$ or $S' \subseteq S_3$, so we have that $S_1$, $S_2$ and $S_3$ are simultaneous metric generators for ${\cal G}\circ{\cal H}$. Thus, 

\begin{displaymath}
\begin{array}{rcl}
\Sd({\cal G} \circ {\cal H})&\le&\min\{\vert S_1 \vert, \vert S_2 \vert, \vert S_3 \vert\}=\\
&=&\vert V_1-V_M({\cal G}) \vert \cdot \Sd_A({\cal H}) +\\
&&+ \vert V_M({\cal G}) \vert \cdot \min \left\{\displaystyle{\min_{B_1 \in {\cal B}_1({\cal H})}}\{\vert B'_1 \vert\},\displaystyle{\min_{B_2 \in {\cal B}_2({\cal H})}}\{\vert B'_2 \vert\}, \underset{\underset{B_2 \in {\cal B}_2({\cal H})}{_{B_1 \in {\cal B}_1({\cal H})}}}{\min} \{\vert B_1 \cup B_2 \vert\}\right\}\leq\\
&\leq&\vert V_1-V_M({\cal G}) \vert \cdot \Sd_A({\cal H}) + \vert V_M({\cal G}) \vert \cdot \left(\Sd_A({\cal H})+\zeta({\cal H})\right)=\\
&=&\vert V_1 \vert \cdot \Sd_A({\cal H}) + \zeta({\cal H}) \cdot \vert V_M({\cal G}) \vert.
\end{array}
\end{displaymath}

As in the previous cases, by exchanging the roles of ${\cal B}_1$ and ${\cal B}_2$ for $\overline{\cal H}$ and proceeding in an analogous manner as above, we obtain that $$\vert V_1 \vert \cdot \Sd_A({\cal H})+\vert V_M({\cal G}) \vert \leq \Sd({\cal G} \circ \overline{\cal H}) \leq \vert V_1 \vert \cdot \Sd_A({\cal H}) + \zeta({\cal H}) \cdot \vert V_M({\cal G}) \vert.$$

The proof is thus complete.
\end{proof}

We now analyse the different cases described in Theorem~\ref{upperBoundFamilyLex1FamGFamH}. First, note that if $\zeta({\cal H})=1$, then Equation~(\ref{eqIneqSdA}) becomes an equality. In particular, $\zeta({\cal H})=1$ for every ${\cal H}=\{H\}$. Additionally, if there exists a simultaneous adjacency basis $B_1 \in {\cal B}_1({\cal H})$ such that one vertex $v \in V_2-B_1$ satisfies $B_1 \cap N_{H}(v)=\emptyset$ for every $H \in {\cal H}$, then $\zeta({\cal H})=1$. In an analogous manner, if there exists a simultaneous adjacency basis $B_2 \in {\cal B}_2({\cal H})$ such that one vertex $v \in V_2-B_2$ satisfies $B_2 \subseteq N_{H}(v)$ for every $H \in {\cal H}$, then $\zeta({\cal H})=1$. Finally, if there exist two simultaneous adjacency bases $B_1 \in {\cal B}_1({\cal H})$ and $B_2 \in {\cal B}_2({\cal H})$ such that $|B_1 \cup B_2| = \Sd_A({\cal H})+1$, then $\zeta({\cal H})=1$.

Next, we discuss Equation~(\ref{eqFirstEqSdA}). First, note that $V_M(\{G\})=\emptyset$ for every graph $G$. Now, we analyse several non-trivial conditions under which a graph family ${\cal G}$ composed by connected graphs on a common vertex set satisfies $V_M({\cal G})=\emptyset$. Consider two vertices $u$ and $v$ that are true twins in some graph $G$, and a vertex $x \in V(G)-\{u,v\}$ such that $x \sim u$ and $x \sim v$. We have that $\langle\{u,v,x\}\rangle_G \cong C_3$. This fact allows us to characterize a large number of families composed by true-twins-free graphs, for which $V_M({\cal G})=\emptyset$.

\begin{remark}\label{cond3VMEmpty}
Let ${\cal G}$ be a graph family on a common vertex set, such that every $G\in {\cal G}$ is a tree or satisfies $\mathtt{g}(G)\ge 4$. Then, $V_M({\cal G})=\emptyset$.
\end{remark}

In particular, for families composed by path or cycle graphs of order greater than or equal to four, not only all members are true-twins-free, but they are also false-twins-free. Moreover, families composed by hypercubes of degree $r \ge 2$ satisfy that all their members have girth four.

We now study the behaviour of $V_M({\cal H})$ for ${\cal H} \subseteq {\cal G}_B(G)$, where $B$ is an adjacency basis of $G$.

\begin{remark}\label{cond1VMEmpty}
For every adjacency basis $B$ of a graph $G$, and every family ${\cal H} \subseteq {\cal G}_B(G)$, $$V_M({\cal H})=\emptyset.$$
\end{remark}

\begin{proof}
Let $B$ be an adjacency basis of $G$. Consider a pair of vertices $x,y \in B$. By the construction of ${\cal G}_B(G)$, we have that in every $H \in {\cal H}$ either $x$ and $y$ are true twins, or they are false twins, or they are not twins. Moreover, since $B$ is a simultaneous adjacency generator for ${\cal H}$, no pair of vertices $x,y \in V(G)-B$ are twins in any $H \in {\cal H}$. Finally, consider two vertices $x \in B$ and $y \in V(G)-B$. If there exist graphs $H_1,H_2,\ldots,H_k \in {\cal H}$ where $N_{H_i}(x)=N_{H_i}(y)$, $i \in \{1,\ldots,k\}$, we have that, by the construction of ${\cal G}_B(G)$, either $x \sim y$ in every $H_i$, $i \in \{1,\ldots,k\}$, or $x \nsim y$ in every $H_i$, $i \in \{1,\ldots,k\}$. Hence, $x$ and $y$ are true twins in every $H_i$, $i \in \{1,\ldots,k\}$, or they are false twins in every $H_i$, $i \in \{1,\ldots,k\}$. In consequence, $V_M({\cal H})=\emptyset$.
\end{proof}

We now discuss several cases where a graph family ${\cal H}$ satisfies ${\cal B}_1({\cal H}) \cap {\cal B}_2({\cal H}) \neq \emptyset$. First, we introduce an auxiliary result.

\begin{lemma}\label{lemmaExistAdjBDominatingCyclePath}
Let $P_n$ and $C_n$ be a path and a cycle graph of order $n \ge 7$. If $n \mod 5 \in \{0,2,4\}$, then there exist adjacency bases of $P_n$ and $C_n$ that are dominating sets.
\end{lemma}

\begin{proof}
In $C_n$, consider the path $v_iv_{i+1}v_{i+2}v_{i+3}v_{i+4}$, where the subscripts are taken modulo $n$, and an adjacency basis $B$. If $v_i,v_{i+2} \in B$ and $v_{i+1} \notin B$, then $\{v_{i+1}\}$ is said to be a 1-gap of $B$. Likewise, if $v_i,v_{i+3} \in B$ and $v_{i+1},v_{i+2} \notin B$, then $\{v_{i+1},v_{i+2}\}$ is said to be a 2-gap of $B$ and if $v_i,v_{i+4} \in B$ and $v_{i+1},v_{i+2},v_{i+3} \notin B$, then $\{v_{i+1},v_{i+2},v_{i+3}\}$ is said to be a 3-gap of $B$. Since $B$ is an adjacency basis of $C_n$, it has no gaps of size 4 or larger and it has at most one 3-gap. Moreover, every 2- or 3-gap must be neighboured by two 1-gaps and the number of gaps of either size is at most $\dim_A(C_n)$. We now differentiate the following cases for $C_n$:

\begin{enumerate}
\item $n=5k$, $k\ge 2$. In this case, $\dim_A(C_n)=2k$ and $n-\dim_A(C_n)=3k$. Since any 2-gap must be neighboured by two 1-gaps, any adjacency basis has at most $k$ 2-gaps. Any set $B$ having exactly $k$ 2-gaps and exactly $k$ 1-gaps is an adjacency basis of $C_n$, as $|B|\ge 2k=\dim_A(C_n)$ and $|(N_{C_n}(x) \cap B) \triangledown (N_{C_n}(y) \cap B)| \ge 1$ for any pair of different vertices $x,y \in V(C_n)-B$. Since the number of vertices of $V(C_n)-B$ belonging to a 1- or 2-gap is $3k=n-|B|$, we deduce that $B$ has no 3-gaps, \textit{i.e.} it is a dominating set.
\item $n=5k+2$, $k\ge 1$. In this case, $\dim_A(C_n)=2k+1$ and $n-\dim_A(C_n)=3k+1$. As in the previous
case, any adjacency basis has at most $k$ 2-gaps. Moreover, any set $B$ having exactly $k$ 2-gaps and exactly $k+1$ 1-gaps is an adjacency basis of $C_n$, and the number of vertices of $V(C_n)-B$ belonging to a 1- or 2-gap is $3k+1=n-|B|$, so $B$ has no 3-gaps, \textit{i.e.} it is a dominating set.
\item $n=5k+4$, $k\ge 1$. In this case, $\dim_A(C_n)=2k+2$ and $n-\dim_A(C_n)=3k+2$. Assume that some adjacency basis $B$ has $k+1$ 2-gaps. Then, $B$ would have at least $k+1$ 1-gaps, making $|V(C_n)-B|\ge 3k+3$, which is a contradiction. So, any adjacency basis has at most $k$ 2-gaps. As in
the previous cases, any set $B$ having exactly $k$ 2-gaps and exactly $k+2$ 1-gaps is an adjacency basis of $C_n$, and the number of vertices of $V(C_n)-B$ belonging to a 1- or 2-gap is $3k+2=n-|B|$, so $B$ has no 3-gaps, \textit{i.e.} it is a dominating set.
\end{enumerate}

By the set of cases above, the result holds for $C_n$.

Now consider the path $P_n$, where $n \mod 5 \in \{0,2,4\}$, and let $C'_n$ be the cycle obtained from $P_n$ by joining its leaves $v_1$ and $v_n$ by an edge. Let $B$ be an adjacency basis of $C'_n$ which is also a dominating set and satisfies $v_1,v_n \notin B$ (at least one such $B$ exists). We have that every $u \in B$ and every $v \in V(P_n)-B$ satisfy $d_{C'_n,2}(u,v)=d_{P_n,2}(u,v)$, so $B$ is also an adjacency basis and a dominating set of $P_n$.

\end{proof}

The following results hold.

\begin{remark}\label{upperBoundFamilyLex1FamGFamH_Example1}
Let $P_n$ be a path graph of order $n \geq 7$, where $n \mod 5 \in \{0,2,4\}$, and let $C_n$ be the cycle graph obtained from $P_n$ by joining its leaves by an edge. Let $B$ be an adjacency basis of $P_n$ and $C_n$ which is also a dominating set of both. Then, every ${\cal H} \subseteq {\cal G}_B(P_n)\cup{\cal G}_B(C_n)$ such that $P_n \in {\cal H}$ or $C_n \in {\cal H}$ satisfies ${\cal B}_1({\cal H}) \cap {\cal B}_2({\cal H}) \neq \emptyset$.
\end{remark}

\begin{proof}
The existence of $B$ is a consequence of Lemma~\ref{lemmaExistAdjBDominatingCyclePath}. Since $P_n \in {\cal H}$ or $C_n \in {\cal H}$, we have that $B$ is a simultaneous adjacency basis of ${\cal H}$. Let $V=V(P_n)=V(C_n)$. By the definition of ${\cal G}_B$, we have that $\displaystyle{\bigcup_{v \in B}}N_{H}(v)=\displaystyle{\bigcup_{v \in B}}N_{P_n}(v)=V$ or $\displaystyle{\bigcup_{v \in B}}N_{H}(v)=\displaystyle{\bigcup_{v \in B}}N_{C_n}(v)=V$ for every $H \in {\cal H}$, so $B$ is a dominating set of every $H \in {\cal H}$. Moreover, by Lemma~\ref{LemmaDiameter>6orPathorCycle}, we have that $B \nsubseteq N_{P_n}(v)$ and $B \nsubseteq N_{C_n}(v)$ for every $v \in V$. Furthermore, by the definition of ${\cal G}_B$, we have that $B \cap N_{H}(v)=B \cap N_{P_n}(v)$ or $B \cap N_{H}(v)=B \cap N_{C_n}(v)$ for every $H \in {\cal H}$ and every $v\in V$, so $B \nsubseteq N_{H}(v)$ for every $H \in {\cal H}$ and every $v\in V$. In consequence, $B \in {\cal B}_1({\cal H}) \cap {\cal B}_2({\cal H})$, so the result holds.
\end{proof}

The following result is a direct consequence of Theorem~\ref{upperBoundFamilyLex1FamGFamH} and Remark~\ref{upperBoundFamilyLex1FamGFamH_Example1}.

\begin{proposition}\label{upperBoundFamilyLex1_Example1}
Let ${\cal G}$ be a family of connected graphs on a common vertex set $V$, let $P_n$ be a path graph of order $n \geq 7$, where $n~\mod~5 \in \{0,2,4\}$, and let $C_n$ be the cycle graph obtained from $P_n$ by joining its leaves by an edge. Let $B$ be an adjacency basis of $P_n$ and $C_n$ which is also a dominating set of both. Then, for every ${\cal H} \subseteq {\cal G}_B(P_n)\cup{\cal G}_B(C_n)$ such that $P_n \in {\cal H}$ or $C_n \in {\cal H}$, $$\Sd({\cal G} \circ {\cal H})=|V|\cdot\left\lfloor\frac{2n+2}{5}\right\rfloor.$$
\end{proposition}

\begin{remark}\label{upperBoundFamilyLex1FamGFamH_Example2}
Let ${\cal H}$ be a graph family on a common vertex set $V$ of cardinality $|V|\ge 7$ such that every $H\in {\cal H}$ is a path graph, a cycle graph, $D(H)\ge 6$, or $\mathtt{g}(H) \ge 5$ and $\delta(H) \ge 3$. Let ${\cal H}'$ be a graph family on a common vertex set $V'$ of cardinality $|V'|\ge 7$ satisfying the same conditions as ${\cal H}$.  Then, ${\cal B}_1({\cal H}+{\cal H}') \cap {\cal B}_2({\cal H}+{\cal H}') \neq \emptyset$.
\end{remark}

\begin{proof}
As we discussed in the proof of Theorem~\ref{SimAdjDimJoinsFamilyCase1}, there exists a simultaneous metric basis $B$ of ${\cal H}+{\cal H}'$, which is also a simultaneous adjacency basis, such that the sets $W=B\cap V$ and $W'=B\cap V'$ satisfy $W \nsubseteq N_{H}(v)$ for every $H \in {\cal H}$ and every $v \in V$, and $W' \nsubseteq N_{H'}(w)$ for every $H' \in {\cal H}'$ and every $w \in V'$. In consequence, we have that $B \nsubseteq N_{H+H'}(v)$ for every $H+H' \in {\cal H}+{\cal H}'$ and every $v \in V \cup V'$. Moreover, every vertex in $V$ is dominated by every vertex in $W'$, whereas every vertex in $V'$ is dominated by every vertex in $W$, so $B$ is a dominating set for every $H+H' \in {\cal H}+{\cal H}'$. In consequence, $B \in {\cal B}_1({\cal H}+{\cal H}') \cap {\cal B}_2({\cal H}+{\cal H}')$, so the result holds.
\end{proof}

By an analogous reasoning, Theorems~\ref{dimAPermFamily} and  \ref{SimAdjDimJoinsFamilyCase1} lead to the next result.

\begin{remark}\label{upperBoundFamilyLex1FamGFamH_Example3}
Let $H$ be a graph of order $n \ge 7$ which is a path graph, or a cycle graph, or satisfies $D(H)\ge 6$, or $\mathtt{g}(H) \ge 5$ and $\delta(H) \ge 3$. Let $H'$ be a graph of order $n' \ge 7$ that satisfies the same conditions as $H$. Let $B$ and $B'$ be adjacency bases of $H$ and $H'$, respectively. Then, any pair of families ${\cal H} \subseteq {\cal G}_B(H)$ and ${\cal H}' \subseteq {\cal G}_{B'}(H')$ such that $H \in {\cal H}$ and $H' \in {\cal H}'$ satisfies ${\cal B}_1({\cal H}+{\cal H}') \cap {\cal B}_2({\cal H}+{\cal H}') \neq \emptyset$.
\end{remark}

The two following results are direct consequences of Theorem~\ref{upperBoundFamilyLex1FamGFamH} and Remarks~\ref{upperBoundFamilyLex1FamGFamH_Example2} and~\ref{upperBoundFamilyLex1FamGFamH_Example3}.

\begin{proposition}\label{upperBoundFamilyLex1_Example2}
Let ${\cal G}$ be a family of connected graphs on a common vertex set $V_1$. Let ${\cal H}$ be a graph family on a common vertex set $V_2$ of cardinality $|V_2|\ge 7$ such that every $H\in {\cal H}$ is a path graph, a cycle graph, $D(H)\ge 6$, or $\mathtt{g}(H) \ge 5$ and $\delta(H) \ge 3$. Let ${\cal H}'$ be a graph family on a common vertex set $V'_2$ of cardinality $|V'_2|\ge 7$ satisfying the same conditions as ${\cal H}$.  Then, $$\Sd({\cal G} \circ ({\cal H}+{\cal H}'))=|V_1| \cdot \Sd_A({\cal H})+|V_1| \cdot \Sd_A({\cal H}').$$
\end{proposition}

\begin{proposition}\label{upperBoundFamilyLex1_Example3}
Let ${\cal G}$ be a family of connected graphs on a common vertex set $V$. Let $H$ be a graph of order $n \ge 7$ which is a path graph, or a cycle graph, or satisfies $D(H)\ge 6$, or $\mathtt{g}(H) \ge 5$ and $\delta(H) \ge 3$. Let $H'$ be a graph of order $n' \ge 7$ that satisfies the same conditions as $H$. Let $B$ and $B'$ be adjacency bases of $H$ and $H'$, respectively. Then, for any pair of families ${\cal H} \subseteq {\cal G}_B(H)$ and ${\cal H}' \subseteq {\cal G}_{B'}(H')$ such that $H \in {\cal H}$ and $H' \in {\cal H}'$, $$\Sd({\cal G} \circ ({\cal H}+{\cal H}'))=|V| \cdot \dim_A(H)+|V| \cdot \dim_A(H').$$
\end{proposition}

We now analyse several conditions under which a graph family ${\cal G}$ composed by connected graphs on a common vertex set satisfies $V_M({\cal G}) \neq \emptyset$ and, in some cases, we exactly determine the value of $V_M({\cal G})$. It is simple to see that any graph of the form $K_t+G$, $t \ge 2$, satisfies $V(K_t) \subseteq v^*$ for some $v^* \in T(K_t+G)$. Likewise, any graph of the form $N_t+G$, $t \ge 2$, satisfies $V(N_t) \subseteq v^*$ for some $v^* \in F(N_t+G)$. Moreover, any complete graph $K_n$, $n \ge 2$, satisfies $T(G)=\{V(K_n)\}$. The next results are direct consequences of these facts.

\begin{remark}\label{cond2VMNotEmpty}
Let ${\cal G}=\{G_1,G_2,\ldots,G_k\}$ be a family of connected graphs on a common vertex set $V$ such that, for some $i \in \{1,\ldots,k\}$, $G_i=N_t+G'$, where $N_t$ is an empty graph on the vertex set $V' \subset V$, $|V'| \ge 2$, and $G'=(V-V',E')$. If, for some $j \in \{1,\ldots,k\}-\{i\}$, $G_j=K_t+G''$, where $K_t$ is a complete graph on the vertex set $V'$ and $G''=(V-V',E'')$, then $V_M({\cal G}) \neq \emptyset$.
\end{remark}

\begin{corollary}\label{cond2VMNotEmptyPartCase1}
Let ${\cal G}=\{G_1,G_2,\ldots,G_k\}$ be a family composed by path or cycle graphs on a common vertex set $V_1$ of size $n \ge 4$, and let $\{K_t,N_t\}$ be a family composed by a complete and an empty graph on a common vertex set $V_2$ of size $t \ge 2$. Then every ${\cal H} \subseteq \{N_t+G_1,N_t+G_2,\ldots,N_t+G_k\}$, ${\cal H}\neq\emptyset$, and every ${\cal H}' \subseteq \{K_{n+t}, K_t+G_1,K_t+G_2,\ldots,K_t+G_k\}$, ${\cal H}'\neq\emptyset$, satisfy $V_M({\cal H} \cup {\cal H}')=V_2$.
\end{corollary}

We now analyse cases of families containing a graph and its complement.

\begin{remark}\label{cond1VMNotEmpty}
Let $G$ be a connected graph such that $|T(G)|\ge 1$ or $|F(G)|\ge 1$, and $\overline{G}$ is connected. Then any family ${\cal G}$ composed by connected graphs on a common vertex set such that $G \in {\cal G}$ and $\overline{G} \in {\cal G}$ satisfies $V_M({\cal G})\neq \emptyset$.
\end{remark}

\begin{proof}
First assume that $|T(G)|\ge 1$. Consider a true twin equivalence class $v_1^*=\{v_1,v_2,\ldots,v_t\} \in T(G)$. For every pair of vertices $v_i,v_j\in v_1^*$, we have that $N_{\overline{G}}(v_i)=N_{\overline{G}}(v_j)$ and $v_i \nsim_{\overline{G}} v_j$. In consequence, $v_1^*$ is a false twin equivalence class of $\overline{G}$. Now assume that $|F(G)|\ge 1$ and consider a false twin equivalence class $w_1^*=\{w_1,w_2,\ldots,w_f\} \in F(G)$. For every pair of vertices $w_i,w_j\in w_1^*$, we have that $N_{\overline{G}}[w_i]=N_{\overline{G}}[w_j]$, so $w_1^*$ is a true twin equivalence class of $\overline{G}$. In consequence, $V_T(G) \cup V_F(G) \subseteq V_M({\cal G})$, so the result follows.
\end{proof}

\begin{corollary}\label{cond1VMNotEmptyPartCase1}
For every connected graph $G$ such that $\overline{G}$ is connected, $V_M(\{G,\overline{G}\})=V_T(G) \cup V_F(G)$.
\end{corollary}

Finally, we analyse some examples of families ${\cal H}$ satisfying ${\cal B}_1({\cal H}) \cap {\cal B}_2({\cal H}) = \emptyset$. Consider the family ${\cal H}_5=\{P_5,C_5\}$, where $V(P_5)=V(C_5)=\{v_1,v_2,v_3,v_4,v_5\}$, $E(P_5)=\{v_1v_2,v_2v_3,v_3v_4,$ $v_4v_5\}$ and $E(C_5)=E(P_5) \cup \{v_1v_5\}$. We have that ${\cal B}_1({\cal H}_5)=\{\{v_1,v_5\},\{v_2,v_3\},$ $\{v_3,v_4\}\}$ and ${\cal B}_2({\cal H}_5)=\{\{v_2,v_4\}\}$, that is ${\cal B}_1({\cal H}_5) \cap {\cal B}_2({\cal H}_5) = \emptyset$. Likewise, ${\cal B}_1(\{P_5\})=\{\{v_1,v_5\},\{v_2,v_3\},\{v_3,v_4\}\}$ and ${\cal B}_2(\{P_5\})=\{\{v_2,v_4\}\}$, \textit{i.e.} ${\cal B}_1(\{P_5\}) \cap {\cal B}_2(\{P_5\}) = \emptyset$; whereas ${\cal B}_1(\{C_5\})=\{\{v_1,v_2\},\{v_1,v_5\},\{v_2,v_3\},$ $\{v_3,v_4\},\{v_4,v_5\}\}$ and ${\cal B}_2(\{C_5\})=\{\{v_1,v_3\},\{v_1,v_4\},\{v_2,v_4\},$ $\{v_2,v_5\},\{v_3,v_5\}\}$, \textit{i.e.} ${\cal B}_1(\{C_5\}) \cap {\cal B}_2(\{C_5\}) = \emptyset$. Moreover, the vertex $v_3$ satisfies $\{v_2,v_4\} \subseteq N_{P_5}(v_3)$ and $\{v_2,v_4\} \subseteq N_{C_5}(v_3)$, so $\zeta({\cal H})=1$ for every non-empty subfamily ${\cal H} \subseteq {\cal H}_5$.

Additionally, consider the family ${\cal H}_{ex}^{(n)}=\{H_1,H_2,H_3,H_4\}$ depicted in Figure~\ref{figAddFamilyH}. ${\cal H}_{ex}^{(n)}$ is defined on the common vertex set $V=\{v_1,\ldots,v_n,v_{n+1},\ldots,v_{n+6}\}$, $n \ge 7$, $n \mod 5 \in \{0,2,4\}$, and the dashed lines in the figure indicate that $H_i$ differs from $H_j$ in the fact of containing, or not, each one of the edges $v_1v_n$ and $v_{n+2}v_{n+4}$. Let $V_1=\{v_1,\ldots,v_n\}$ and $V_2=\{v_{n+1},\ldots,v_{n+6}\}$. We have that, for every $H \in {\cal H}_{ex}^{(n)}$, $\langle V_1 \rangle_H \cong P_n$ or $\langle V_1 \rangle_H \cong C_n$.
In consequence, for every non-empty subfamily ${\cal H} \subseteq {\cal H}_{ex}^{(n)}$, we have that $\Sd_A({\cal H})=\dim_A(P_n)+2=\dim_A(C_n)+2$, and every simultaneous adjacency basis $B$ has the form $B=B' \cup X$, where $X \subset V_2$ and $B'$ is a simultaneous adjacency basis of ${\cal H}'=\{\langle V_1 \rangle_{H} :\; H  \in {\cal H}\}$. Moreover, we have that ${\cal B}_1({\cal H})=\{B' \cup X\}$, where $B'$ is a simultaneous adjacency basis of ${\cal H}'$ that is also a dominating set of every $H' \in {\cal H}'$ (Lemma~\ref{lemmaExistAdjBDominatingCyclePath}, and the fact that two graphs in ${\cal H}'$ differ at most in the fact of containing, or not, the edge $v_1v_n$, guarantee the existence of such $B'$) and $X \in \{\{v_{n+2},v_{n+3}\},\{v_{n+3},v_{n+4}\},\{v_{n+3},v_{n+5}\},\{v_{n+3},v_{n+6}\},\{v_{n+5},v_{n+6}\}\}$. Likewise, ${\cal B}_2({\cal H})=\{B' \cup \{v_{n+2},v_{n+4}\}\}$, where $B'$ is a simultaneous adjacency basis of ${\cal H}'$ that is also a dominating set of every $H' \in {\cal H}'$. Clearly, ${\cal B}_1({\cal H}) \cap {\cal B}_2({\cal H})=\emptyset$. Moreover, for every $B \in {\cal B}_2({\cal H})$, the vertex $v_{n+1}$ satisfies $B \subseteq N_{H}(v_{n+1})$ for every $H \in {\cal H}$, so $\zeta({\cal H})=1$.

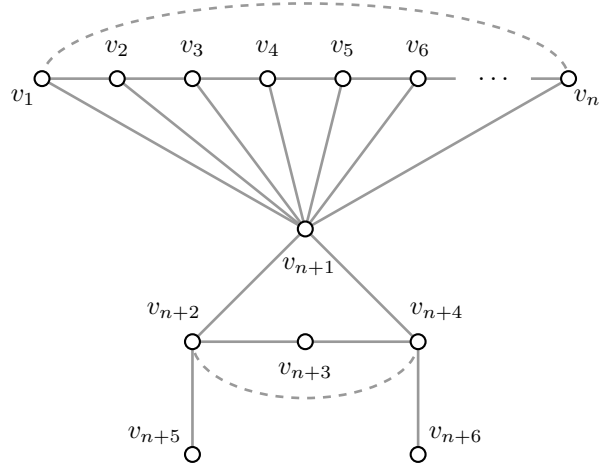
\begin{figure}[h]
\begin{center}
\begin{tikzpicture}[inner sep=0.7mm, place/.style={circle,draw=black,
fill=white,thick},dxx/.style={circle,draw=black!99,fill=black!99,thick},lxx/.style={circle,draw=black!99,densely dotted,fill=black!60,thick},transition/.style={rectangle,draw=black!50,fill=black!20,thick},line width=1pt,scale=0.5]

\coordinate (A) at (4,0);
\coordinate (B) at (10,0);
\coordinate (C) at (4,3);
\coordinate (D) at (7,3);
\coordinate (E) at (10,3);
\coordinate (F) at (7,6);

\coordinate (M) at (0,10);
\coordinate (N) at (2,10);
\coordinate (O) at (4,10);
\coordinate (P) at (6,10);
\coordinate (Q) at (8,10);
\coordinate (R) at (10,10);
\coordinate (R1) at (11,10);
\coordinate (S1) at (13,10);
\coordinate (S) at (14,10);

\draw[black!40] (A) -- (C) -- (D) -- (E) -- (B);
\draw[black!40] (C) -- (F) -- (E);
\draw[dashed,black!40] (E) arc (360:180:3 and 1.5);
\draw[black!40] (M) -- (N) -- (O) -- (P) -- (Q) -- (R) -- (R1);
\draw[black!40] (S1) -- (S);
\draw[dashed,black!40] (M) arc (180:0:7 and 2);
\draw[black!40] (F) -- (M);
\draw[black!40] (F) -- (N);
\draw[black!40] (F) -- (O);
\draw[black!40] (F) -- (P);
\draw[black!40] (F) -- (Q);
\draw[black!40] (F) -- (R);
\draw[black!40] (F) -- (S);

\node at (A) [place]  {};
\coordinate [label=center:{$v_{n+5}$}] (v17) at (3,0.5);
\node at (B) [place]  {};
\coordinate [label=center:{$v_{n+6}$}] (v14) at (11,0.5);
\node at (C) [place]  {};
\coordinate [label=center:{$v_{n+2}$}] (v18) at (3.5,3.8);
\node at (D) [place]  {};
\coordinate [label=center:{$v_{n+3}$}] (v12) at (7,2.2);
\node at (E) [place]  {};
\coordinate [label=center:{$v_{n+4}$}] (v15) at (10.5,3.8);
\node at (F) [place]  {};
\coordinate [label=center:{$v_{n+1}$}] (v16) at (7.1,5);

\node at (M) [place]  {};
\coordinate [label=center:{$v_1$}] (v24) at (-0.5,9.5);
\node at (N) [place]  {};
\coordinate [label=center:{$v_2$}] (v21) at (2,10.8);
\node at (O) [place]  {};
\coordinate [label=center:{$v_3$}] (v25) at (4,10.8);
\node at (P) [place]  {};
\coordinate [label=center:{$v_4$}] (v22) at (6,10.8);
\node at (Q) [place]  {};
\coordinate [label=center:{$v_5$}] (v32) at (8,10.8);
\node at (R) [place]  {};
\coordinate [label=center:{$v_6$}] (v38) at (10,10.8);
\coordinate [label=center:{$\ldots$}] (dots1) at (12,10);
\node at (S) [place]  {};
\coordinate [label=center:{$v_n$}] (v33) at (14.5,9.5);

\end{tikzpicture}

\end{center}
\caption{For $n \ge 7$, $n \mod 5 \in \{0,2,4\}$, every non-empty subfamily ${\cal H}$ of the family ${\cal H}_{ex}^{(n)}=\{H_1,H_2,H_3,H_4\}$ satisfies ${\cal B}_1({\cal H}) \cap {\cal B}_2({\cal H})=\emptyset$ and $\zeta({\cal H})=1$.}
\label{figAddFamilyH}
\end{figure}

The aforementioned facts, along with Corollaries~\ref{cond2VMNotEmptyPartCase1} and~\ref{cond1VMNotEmptyPartCase1}, allows us to obtain examples where Equation~(\ref{eqIneqSdA}) becomes an equality.

\begin{proposition}\label{secondEqualityLexFamGFamH_expl}
Let ${\cal G}=\{G_1,G_2,\ldots,G_k\}$ be a family composed by path or cycle graphs on a common vertex set $V_1$ of size $p \ge 4$, and let $\{K_t,N_t\}$ be a family composed by a complete and an empty graph on a common vertex set $V_2$ of size $t \ge 2$. Let ${\cal G}' \subseteq \{N_t+G_1,N_t+G_2,\ldots,N_t+G_k\}$, ${\cal G}'\neq\emptyset$, and let ${\cal G}'' \subseteq \{K_{n+t}, K_t+G_1,K_t+G_2,\ldots,K_t+G_k\}$, ${\cal G}''\neq\emptyset$. Then, the following assertions hold:
\begin{enumerate}[{\rm (i)}]
\item For every non-empty subfamily ${\cal H} \subseteq {\cal H}_5$, $$\Sd(({\cal G}' \cup {\cal G}'') \circ {\cal H})=|V_1 \cup V_2|\cdot\Sd_A({\cal H})+|V_M({\cal G}' \cup {\cal G}'')|=2p+3t.$$
\item For every $n \ge 7$, where $n \mod 5 \in \{0,2,4\}$, and every non-empty subfamily ${\cal H} \subseteq {\cal H}_{ex}^{(n)}$, $$\Sd(({\cal G}' \cup {\cal G}'')\circ {\cal H})=|V_1 \cup V_2|\cdot\Sd_A({\cal H})+|V_M({\cal G}' \cup {\cal G}'')|=(p+t)\cdot\left(\left\lfloor\frac{2n+2}{5}\right\rfloor+2\right)+t.$$
\end{enumerate}
\end{proposition}

\begin{proposition}
Let $G$ be a connected graph of order $q$ such that $\overline{G}$ is connected. Then, the following assertions hold:
\begin{enumerate}[{\rm (i)}]
\item For every non-empty subfamily ${\cal H} \subseteq {\cal H}_5$, $$\Sd(\{G,\overline{G}\} \circ {\cal H})=q\cdot\Sd_A({\cal H})+|V_M(\{G,\overline{G}\})|=2q+|V_T(G)|+|V_F(G)|.$$
\item For every $n \ge 7$, where $n \mod 5 \in \{0,2,4\}$, and every non-empty subfamily ${\cal H} \subseteq {\cal H}_{ex}^{(n)}$, $$\Sd(\{G,\overline{G}\}\circ {\cal H})=q\cdot\Sd_A({\cal H})+|V_M(\{G,\overline{G}\})|=q\cdot\left(\left\lfloor\frac{2n+2}{5}\right\rfloor+2\right)+|V_T(G)|+|V_F(G)|.$$
\end{enumerate}
\end{proposition}

The previous examples additionally show that the bounds of Equation~(\ref{eqIneqSdA}) are tight. In general, the upper bound is reached when $\min\{\vert S_1 \vert,\vert S_2 \vert,\vert S_3 \vert\}=\vert S_3 \vert$ or when for every $B_1 \in {\cal B}_1({\cal H})$ there exist exactly $k_2$ vertices $v_{p_1},v_{p_2},\ldots,v_{p_r} \in V_2-B_1$ such that $B_1 \cap N_{H}(v_{p_i})=\emptyset$ for some $H \in {\cal H}$ and for every $B_2 \in {\cal B}_2({\cal H})$ there exist exactly $k_2$ vertices $v_{q_1},v_{q_2},\ldots,v_{q_s} \in V_2-B_2$ such that $B_2 \subseteq N_{H}(v_{q_i})$ for some $H \in {\cal H}$.

In order to present our next results, we introduce some additional definitions. For a family ${\cal H}$ of non-trivial graphs on a common vertex set $V$, and a simultaneous adjacency basis $B \in {\cal B}({\cal H})$, consider the sets $$P(B)=\{v \in V :\; B \subseteq N_{H}(v) \textrm{ for some } H \in {\cal H}\}$$ and $$Q(B)=\{v \in V :\; B \cap N_{H}(v)=\emptyset \textrm{ for some } H \in {\cal H}\}.$$

Based on the definitions of $P(B)$ and $Q(B)$, we define the parameter $$\xi(G,{\cal H})=\underset{B \in {\cal B}({\cal H})}{\min}\left\{\vert P(B)\vert\left(\vert V_T(G)\vert -\vert T(G)\vert\right)+\vert Q(B)\vert \left(\vert V_F(G)\vert -\vert F(G)\vert\right)\right\}.$$

Finally, for a graph $G$, let $V'_T(G)=\bigcup_{v^* \in T(G)}(v^*-\{v\})$ be the set composed by all vertices, except one, from every true twin equivalence class of $G$. Likewise, let $V'_F(G)=\bigcup_{v^* \in F(G)}(v^*-\{v\})$ be the set composed by all vertices, except one, from every false twin equivalence class of $G$. For convenience, we will assume without loss of generality that for every graph $G$ a fixed vertex will always be the one excluded from every true or false twin equivalence class when constructing $V'_T(G)$ or $V'_F(G)$, respectively. With these definitions in mind, we give our next result.

\begin{theorem}\label{upperBoundFamilyLex2}
Let $G$ be a connected graph of order $n$ and let ${\cal H}=\{H_1,H_2,\ldots,H_k\}$ be a family of non-trivial graphs on a common vertex set $V_2$. If for every simultaneous adjacency basis $B$ of ${\cal H}$ there exists $H \in {\cal H}$ where one vertex $v$ satisfies $B \subseteq N_{H}(v)$, or there exists $H' \in {\cal H}$ for which $B$ is not a dominating set, then $$n\cdot\Sd_A({\cal H})\leq\Sd(G \circ {\cal H}) \le n\cdot\Sd_A({\cal H})+ \xi(G,{\cal H}).$$
\end{theorem}

\begin{proof}
$\Sd(G \circ {\cal H})\geq n\cdot\Sd_A({\cal H})$ by Theorem~\ref{lowerBoundSdLexSdAH}, so we only need to prove that $\Sd(G \circ {\cal H}) \le n\cdot\Sd_A({\cal H})+ \xi(G,{\cal H})$. Let $B$ be a simultaneous adjacency basis of ${\cal H}$ for which $\xi(G,{\cal H})$ is obtained. We differentiate the following cases for every graph $H_i \in {\cal H}$:
\begin{enumerate}
\item There exist $w_1,w_2 \in V_2$ such that $B \subseteq N_{H_i}(w_1)$ and $B \cap N_{H_i}(w_2)=\emptyset$. In this case, we define the set $S_i=(V(G) \times B)\cup(V'_T(G) \times \{w_1\})\cup(V'_F(G) \times \{w_2\})$.

\item There exists $w_1 \in V_2$ such that $B \subseteq N_{H_i}(w_1)$ and there exists no vertex $x\in V_2$ such that $B \cap N_{H_i}(x)=\emptyset$. In this case, we define the set $S_i=(V(G) \times B) \cup (V'_T(G) \times \{w_1\})$.

\item There exists $w_2 \in V_2$ such that $B \cap N_{H_i}(w_2)=\emptyset$ and there exists no vertex $x\in V_2$ such that $B \subseteq N_{H_i}(x)$. In this case, we  define the set $S_i=(V(G) \times B) \cup (V'_F(G) \times \{w_2\})$.

\item There exists no vertex $x\in V_2$ such that $B \subseteq N_{H_i}(x)$ or $B \cap N_{H_i}(x)=\emptyset$. In this case, we define the set $S_i=V(G) \times B$.
\end{enumerate}

For cases 1, 2 and 3, it is shown in \cite{JanOmo2012} that the corresponding set $S_i$ is a metric generator for $G \circ H_i$. Moreover, as we discussed in the proof of Theorem~\ref{upperBoundFamilyLex1FamGFamH}, in case 4 the corresponding set $S_i$ is a metric generator for $G \circ H_i$. In consequence, the set $S=\underset{1\leq i \leq k}{\bigcup}S_i$ is a simultaneous metric generator for $G \circ {\cal H}$. Therefore, $\Sd(G \circ {\cal H}) \leq \vert S \vert = n\cdot\Sd_A({\cal H})+ \xi(G,\mathcal{H})$, so the result holds.
\end{proof} 

The bounds of the inequalities in Theorem~\ref{upperBoundFamilyLex2} are tight. As pointed out in \cite{JanOmo2012}, a twins-free graph $G$ satisfies $T(G)=V_T(G)=F(G)=V_F(G)=\emptyset$. In consequence, $\xi(G,{\cal H})=0$ for any twins-free graph $G$ and any graph family ${\cal H}$, so Theorem~\ref{upperBoundFamilyLex2} leads to the next result.

\begin{proposition}
Let $G$ be a twins-free connected graph of order $n$, and let ${\cal H}$ be a family of non-trivial graphs on a common vertex set. Then, $$\Sd(G\circ{\cal H})=n\cdot\Sd_A({\cal H}).$$
\end{proposition}

Recall the families ${\cal K}(V)$ of star graphs defined in Section~\ref{sectBasicResults}. The following result is an example of a family for which the upper bound of the inequalities of Theorem~\ref{upperBoundFamilyLex2} is reached.

\begin{proposition}\label{partCaseP2CircStars}
For every finite set $V$ of size $|V|\ge 4$, $\Sd(P_2 \circ {\cal K}(V))=2\cdot|V|-1$.
\end{proposition}

\begin{proof}
By Corollary~\ref{partCaseStars}, every simultaneous adjacency basis $B$ of ${\cal K}(V)$ has the form $V-\{v_i\}$, $i \in \{1,\ldots,|V|\}$. In $K_{1,n-1}^i$, we have that $B \subseteq N_{K_{1,n-1}^i}(v_i)$, so $\xi(P_2,{\cal K}(V))=1$. Thus, $\Sd(P_2 \circ {\cal K}(V)) \leq 2 \cdot \Sd_A({\cal K}(V))+1=2 \cdot |V|-1$. Additionally, since $P_2 \circ H \cong H+H$ for any graph $H$, we have that $\Sd(P_2 \circ {\cal K}(V))=\Sd({\cal K}(V)+{\cal K}(V)) \ge 2\cdot\Sd_A({\cal K}(V))+1=2\cdot|V|-1$ by Theorem~\ref{SimAdjDimJoinsFamilyCase2}, so the equality holds.
\end{proof}

As we did for join graphs, now we define large families composed by subgraphs of a lexicographic product graph $G \circ H$, which may be seen as the result of a relaxation of the lexicographic product operation, in the sense that not every pair of nodes from two copies of the second factor corresponding to adjacent vertices of the first factor must be linked by an edge. Since for any adjacency basis $B$ of $G \circ H$, the family ${\cal R}_B$ defined in the next result is a subfamily of ${\cal G}_B(G \circ H)$, the result follows directly from Theorem \ref{dimAPermFamily}.

\begin{corollary}\label{famLexPrEdgesRemoved}
Let $G$ be a connected graph of order $n$, let $H$ be a non-trivial graph and let $B$ be an adjacency basis of $G \circ H$. Let $E'=\{(u_i,u_j)(u_r,u_s) \in E(G \circ H) :\; i \neq r, (u_i,u_j) \notin B, (u_r,u_s) \notin B\}$ and let ${\cal R}_B=\{R_1,R_2,\ldots,R_k\}$ be a graph family, defined on the common vertex set $V(G \circ H)$, such that, for every $l \in \{1,\ldots,k\}$, $E(R_l)=E(G \circ H)-E_l$, for some edge subset $E_l \subseteq E'$. Then $$\Sd({\cal R}_B) \leq \dim(G \circ H).$$
\end{corollary}

\section{Concluding remarks}

In this paper we introduced the notion of  simultaneous adjacency dimension of graph families. We studied its properties, as well as its relationship to the simultaneous metric dimension. In particular, we obtained the exact value, or sharp bounds, for this parameter on some families. Additionally, we showed how, given an adjacency basis $B$ of a graph, large graph families can be constructed having $B$ as a simultaneous adjacency basis.

The simultaneous adjacency dimension proved useful for studying the simultaneous  metric dimension of families composed by lexicographic product graphs. From a general definition of lexicographic product, we focused on two particular cases, namely join graphs  and standard lexicographic product graphs. We obtained relations between the simultaneous metric dimension of families composed by lexicographic product graphs and the simultaneous adjacency dimension of families composed by the second factors. In several cases, these relations allowed us to obtain the exact value of the simultaneous metric dimension for a large number of graph families composed by lexicographic product graphs. 

\bibliographystyle{elsart-num-sort}

\begin{thebibliography}{10}
\expandafter\ifx\csname url\endcsname\relax
  \def\url#1{\texttt{#1}}\fi
\expandafter\ifx\csname urlprefix\endcsname\relax\def\urlprefix{URL }\fi

\bibitem{Brigham2003}
R.~C. Brigham, G.~Chartrand, R.~D. Dutton, P.~Zhang, Resolving domination in
  graphs, Mathematica Bohemica 128~(1) (2003) 25--36.
\newline\urlprefix\url{http://mb.math.cas.cz/mb128-1/3.html}

\bibitem{Brigham1990}
R.~C. Brigham, R.~D. Dutton, Factor domination in graphs, Discrete Mathematics
  86~(1--3) (1990) 127--136.
\newline\urlprefix\url{http://www.sciencedirect.com/science/article/pii/0012365X9090355L}

\bibitem{Chartrand2003}
G.~Chartrand, V.~Saenpholphat, P.~Zhang, The independent resolving number of a
  graph, Mathematica Bohemica 128~(4) (2003) 379--393.
\newline\urlprefix\url{http://mb.math.cas.cz/mb128-4/4.html}

\bibitem{Estrada-Moreno2014a}
A.~Estrada-Moreno, Y.~Ram\'irez-Cruz, J.~A. Rodr\'iguez-Vel\'azquez, On the
  adjacency dimension of graphs, arXiv:1501.04647 [math.CO].
\newline\urlprefix\url{http://arxiv.org/pdf/1501.04647v1.pdf}

\bibitem{Estrada-Moreno2013}
A.~Estrada-Moreno, J.~A. Rodr\'{\i}guez-Vel\'{a}zquez, I.~G. Yero, The
  $k$-metric dimension of a graph, Applied Mathematics \& Information Sciences.
  To appear.
\newline\urlprefix\url{http://arxiv.org/abs/1312.6840}

\bibitem{Estrada-Moreno2013corona}
A.~Estrada-Moreno, I.~G. Yero, J.~A. Rodr{\'i}guez-Vel{\'a}zquez, The
  $k$-metric dimension of corona product graphs, Bulletin of the Malaysian
  Mathematical Sciences Society. To appear.
\newline\urlprefix\url{http://math.usm.my/bulletin/pdf/acceptedpapers/2014-01-033-R1.pdf}

\bibitem{Rodriguez-Velazquez-Fernau2013}
H.~{Fernau}, J.~A. {Rodr{\'{\i}}guez-Vel{\'a}zquez}, On the (adjacency) metric
  dimension of corona and strong product graphs and their local variants:
  combinatorial and computational results, arXiv:1309.2275 [math.CO].
\newline\urlprefix\url{http://arxiv-web3.library.cornell.edu/abs/1309.2275}

\bibitem{Fernau-Ja-Corona-2014}
H.~Fernau, J.~A. Rodr\'{\i}guez-Vel\'{a}zquez, Notions of metric dimension of
  corona products: Combinatorial and computational results, in: Computer
  Science - Theory and Applications, vol. 8476 of Lecture Notes in Computer
  Science, Springer International Publishing, 2014, pp. 153--166.
\newline\urlprefix\url{http://dx.doi.org/10.1007/978-3-319-06686-8_12}

\bibitem{Hammack2011}
R.~Hammack, W.~Imrich, S.~Klav{\v{z}}ar, Handbook of product graphs, Discrete
  Mathematics and its Applications, 2nd ed., CRC Press, 2011.
\newline\urlprefix\url{http://www.crcpress.com/product/isbn/9781439813041}

\bibitem{Harary1976}
F.~Harary, R.~A. Melter, On the metric dimension of a graph, Ars Combinatoria 2
  (1976) 191--195.
\newline\urlprefix\url{http://www.ams.org/mathscinet-getitem?mr=0457289}

\bibitem{IMRAN2013}
M.~Imran, S.~A. ul~Haq~Bokhary, A.~Ahmad,
  A.~Semani\v{c}ov\'a-Fe\v{n}ov\v{c}\'{\i}kov\'a, On classes of regular graphs
  with constant metric dimension, Acta Mathematica Scientia 33~(1) (2013)
  187--206.
\newline\urlprefix\url{http://www.sciencedirect.com/science/article/pii/S0252960212602045}

\bibitem{JanOmo2012}
M.~Jannesari, B.~Omoomi, The metric dimension of the lexicographic product of
  graphs, Discrete Mathematics 312~(22) (2012) 3349--3356.
\newline\urlprefix\url{http://www.sciencedirect.com/science/article/pii/S0012365X12003317}

\bibitem{Johnson1993}
M.~Johnson, Structure-activity maps for visualizing the graph variables arising
  in drug design, Journal of Biopharmaceutical Statistics 3~(2) (1993)
  203--236, pMID: 8220404.
\newline\urlprefix\url{http://www.tandfonline.com/doi/abs/10.1080/10543409308835060}

\bibitem{Johnson1998}
M.~Johnson, Browsable structure-activity datasets, in: R.~Carb\'{o}-Dorca,
  P.~Mezey (eds.), Advances in Molecular Similarity, chap.~8, JAI Press Inc,
  Stamford, Connecticut, 1998, pp. 153--170.
\newline\urlprefix\url{http://books.google.es/books?id=1vvMsHXd2AsC}

\bibitem{Khuller1996}
S.~Khuller, B.~Raghavachari, A.~Rosenfeld, Landmarks in graphs, Discrete
  Applied Mathematics 70~(3) (1996) 217--229.
\newline\urlprefix\url{http://www.sciencedirect.com/science/article/pii/0166218X95001062}

\bibitem{Okamoto2010}
F.~Okamoto, B.~Phinezy, P.~Zhang, The local metric dimension of a graph,
  Mathematica Bohemica 135~(3) (2010) 239--255.
\newline\urlprefix\url{http://dml.cz/dmlcz/140702}

\bibitem{Ramirez-Cruz-Rodriguez-Velazquez_2014}
Y.~Ram\'irez-Cruz, O.~R. Oellermann, J.~A. Rodr\'iguez-Vel\'azquez, The
  simultaneous metric dimension of graph families. Submitted. 
\newline\urlprefix\url{http://arxiv.org/abs/1501.00565}

\bibitem{Ramirez2014}
Y.~Ram\'irez-Cruz, O.~R. Oellermann, J.~A. Rodr\'iguez-Vel\'azquez,
  Simultaneous resolvability in graph families, Electronic Notes in Discrete
  Mathematics 46~(0) (2014) 241 -- 248.
\newline\urlprefix\url{http://www.sciencedirect.com/science/article/pii/S157106531400033X}

\bibitem{Sebo2004}
A.~Seb\"{o}, E.~Tannier, On metric generators of graphs, Mathematics of
  Operations Research 29~(2) (2004) 383--393.
\newline\urlprefix\url{http://dx.doi.org/10.1287/moor.1030.0070}

\bibitem{Slater1975}
P.~J. Slater, Leaves of trees, Congressus Numerantium 14 (1975) 549--559.

\end{thebibliography}

\end{document}